\documentclass{siamart190516}
\usepackage{geometry}

\title{Perturbations of CUR Decompositions}
\headers{Perturbations of CUR Decompositions}{Keaton Hamm and Longxiu Huang}

\usepackage{tablefootnote}
\usepackage{graphicx}
\usepackage{subcaption}
\usepackage{dsfont}
\usepackage{amssymb}
\usepackage{enumerate}
\usepackage{graphicx}
\usepackage{float}
\usepackage{bbm}
\usepackage{amsmath}
\usepackage{comment}
\usepackage{hyperref}
\usepackage{listings}
\usepackage{color}
\usepackage{ulem}
\usepackage{xcolor}

\newsiamremark{remark}{Remark}
\newsiamremark{experiment}{Experiment}

\newcommand{\R}{\mathbb{R}}

\newcommand{\K}{\mathbb{K}}

\newcommand{\C}{\mathbb{C}}

\newcommand{\rank}{{\rm rank\,}}

\newcommand{\eps}{\varepsilon}

\author{Keaton Hamm\thanks{Department of Mathematics, University of Texas at Arlington, Arlington, TX
(\email{keaton.hamm@uta.edu}).}
\and Longxiu Huang\thanks{Department of Mathematics,
University of California Los Angeles, Los Angeles, CA (\email{huangl3@math.ucla.edu}).}}

\begin{document}


\maketitle


\begin{abstract}
The CUR decomposition is a factorization of a low-rank matrix obtained by selecting certain column and row submatrices of it.  We perform a thorough investigation of what happens to such decompositions in the presence of noise.  Since CUR decompositions are non-uniquely formed, we investigate several variants and give perturbation estimates for each in terms of the magnitude of the noise matrix in a broad class of norms which includes all Schatten $p$--norms.  The estimates given here are qualitative and illustrate how the choice of columns and rows affects the quality of the approximation, and additionally we obtain new state-of-the-art bounds for some variants of CUR approximations.
\end{abstract}

\begin{keywords}
CUR Decomposition, Low-Rank Matrix Approximation, Matrix Perturbation, Nystr\"{o}m Method
\end{keywords}
\begin{AMS}
15A23, 65F30, 68P99, 68W20
\end{AMS}


\section{Introduction}

Low-rank matrix approximation has become a mainstay of applied mathematics in recent years, finding applications in signal processing \cite{CandesRomberg}, data compression \cite{fazel2008compressed}, matrix completion \cite{candes2009exact}, and analysis of large-scale data \cite{tropp2019streaming}, to name but a few.  Indeed, it has been observed for some time that much of the data we collect is approximately low rank (see \cite{UdellTownsend2019LowRank} for a prolonged discussion) and thus this structure has been much exploited.  One method for doing so is the CUR decomposition, which while known since at least the 1950s, has recently received much more attention following the works of Goreinov et al.~\cite{Goreinov2,Goreinov,Goreinov3}, and Drineas et al.~\cite{DKMIII,DMM08,DMPNAS}, among others \cite{DemanetWu,SorensenDEIMCUR,VoroninMartinsson} (see \cite{AHKS,HH2019} for a more detailed history of its use).

Classical low-rank matrix approximation methods arose from the Singular Value Decomposition (SVD), while more recent methods typically solve penalized optimization problems \cite{liu2012robust} or use randomized methods in some fashion \cite{DKMIII,DMM08,tropp} due to the lack of robustness of the SVD to noise in many applications \cite{AHKS}, but also due to lack of interpretability of results \cite{DMPNAS}.  

{\color{black}}

\subsection{Contributions}

The main contribution of this work is to provide a thorough perturbation analysis of many different CUR approximations.  The classical CUR decomposition of a low-rank matrix is to put $A=CU^\dagger R$, where $C$ and $R$ are column and row submatrices of $A$, respectively, i.e., $C=A(:,J)$ and $R=A(I,:)$ for some index sets $I,J$, and $U$ is their overlap ($U=A(I,J)$).  {\color{black}Another option, as discussed later, is} $A= CC^\dagger AR^\dagger R$, where $CC^\dagger$ and $R^\dagger R$ are orthogonal projections onto the span of the columns of $C$ and rows of $R$, respectively. We analyze what happens when we observe $\widetilde{A}=A+E$ where $A$ is exactly low rank, and $E$ is an arbitrary noise matrix.  Our estimates in Section \ref{SEC:Main} are qualitative and reminiscent of Stewart's classical perturbation analyses for the SVD and Moore--Penrose pseudoinverses \cite{Stewart_1977,stewart1998perturbation}.  

The main advantages of our approach are as follows:
\begin{itemize}
    \item Our approximation bounds are typically for arbitrary noise matrices $E$, though we specialize in some cases to give tighter and more illustrative bounds,
    \item Our bounds hold for arbitrary Schatten $p$--norms (and even a broader class of norms -- see Section \ref{SEC:Notation} for details),
    \item We show that there is no canonical way to enforce the rank when dealing with CUR decompositions, which is in contrast to what has recently become known about rank-enforcement in the Nystr\"{o}m method, which is a special case of CUR (Section \ref{SEC:Rank}),
    \item In the case of choosing maximal volume submatrices of $\widetilde{A}$, we provide better bounds than those known from the Linear Algebra literature \cite{Osinsky2018} as well as giving an alternate, more intuitive method of proof for the perturbation bound.
    \item {\color{black}We also present a new rank-truncation method for CUR decompositions which demonstrates competitive experimental performance to the state-of-the-art.}
\end{itemize}

{\color{black}\subsection{Modelling Noise} The low-rank + noise model for data is commonly used when considering data obtained from an application domain.  For example, the landmark paper by Cand\`{e}s et al.~on Robust PCA \cite{candes2011robust} considers the case when $\widetilde{A}=A+S$ for low-rank $A$ and sparse $S$ (later works often consider $A+S+E$ where $S$ is sparse, and $E$ has small Frobenius norm).  Just as traditional PCA finds great utility in many data analysis tasks, Robust PCA has achieved success in many areas including image and video processing, medical imaging, and various computer vision tasks; see, e.g., the survey of applications \cite{bouwmans2018applications}.  

Additionally, any matrix which is approximately low-rank can be viewed as $\widetilde{A} = \widetilde{A}_r+(\widetilde{A}-\widetilde{A}_r)$, where the noise matrix satisfies $\|\widetilde{A}-\widetilde{A}_r\|_2 = \sigma_{r+1}(\widetilde{A}).$ Here $\widetilde{A}_r$ is its truncated SVD of order $r$. That is to say, if the spectrum of any matrix $\widetilde{A}$ is not too flat, then it can reasonably be viewed as an instance of the low-rank + noise model and our perturbation analysis may be applied.  Decaying spectrum is often the case in many applications \cite{tropp2019streaming,UdellTownsend2019LowRank}.

Due to this last observation, our results could, in principle, be compared to the substantial literature on the use of CUR decompositions to efficiently and accurately approximate the SVD of a full-rank matrix $\widetilde{A}$, see, e.g., \cite{DKMIII,DMM08,DMPNAS,Wang2019} for a subsampling. However, the focus here is to provide generic bounds for all kinds of noise $E$ without reference to a particular method of selecting columns and rows.  Our aim is to provide flexible and general estimates which can be broadly applied.
}

\subsection{Notations}\label{SEC:Notation}

We will use $\K$ to be either $\R$ or $\C$, and $[n]$ to denote $\{1,\dots,n\}$.  As column-row factorizations choose submatrices of a given matrix, if $A\in\K^{m\times n}$ and $I\subset[m]$, $J\subset[n]$, we let $A(I,J)$ denote the $|I|\times|J|$ submatrix of $A$ with entries $\{a_{i,j}\}_{(i,j)\in I\times J}$, and use $A(I,:)$ to be the case $J=[n]$ and $A(:,J)$ the case $I=[m]$.

We denote by $A=W\Sigma V^*$ (or $W_A\Sigma_AV_A^*$ if the matrix needs to be specified) the Singular Value Decomposition (SVD) of $A$, with the use of $W$ rather than the typical $U$ on account of the latter being used for the middle matrix in the CUR decomposition.  The truncated SVD of order $r$ of a matrix $A$ will be denoted by $A_r=W_r\Sigma_rV_r^*$, where the columns of $W_r$ are the first $r$ left singular vectors, $\Sigma_r$ is a $r\times r$ matrix containing the largest $r$ singular values, and the columns of $V_r$ are the first $r$ right singular vectors.  Singular values are assumed to be positioned in descending order, and we label them $\sigma_1\geq\sigma_2\geq\dots\geq\sigma_k\geq0$, where $k=\rank(A)$. 
If $r=k$, then $A=A_k$ and the truncated SVD $A=W_k\Sigma_k V_k^*$ is also called the compact SVD of $A$. To specify the underlying matrix, we may write $\sigma_i(A)$ for the $i$--th singular value of $A$.  We will also make use of thresholding singular values of a matrix, and will denote by $[A]_\tau$ the matrix $W[\Sigma]_\tau V^*$, where $[\Sigma]_\tau(i,i) = \sigma_i(A)$ if $\sigma_i(A)\geq\tau$, and is $0$ otherwise; thus the case $\tau=0$ corresponds to the full SVD of $A$.

The Moore--Penrose pseudoinverse of $A\in\K^{m\times n}$ is denoted by $A^\dagger\in\K^{n\times m}$.  Recall that this pseudoinverse is unique and satisfies the following properties: (i) $AA^\dagger A = A$, (ii) $A^\dagger AA^\dagger = A^\dagger$, and (iii) $AA^\dagger$ and $A^\dagger A$ are Hermitian.  Additionally, the Moore--Penrose pseudoinverse admits an easy expression given the SVD: $A^\dagger=V_A\Sigma_A^\dagger W_A^*$, where $\Sigma^\dagger$ is the $n\times m$ matrix with diagonal entries $\frac{1}{\sigma_i(A)}$, $i=1,\dots,\rank(A)$. 

In our analysis we consider a general family of matrix norms as in Stewart \cite{Stewart_1977}.  The spectral norm is denoted by $\|A\|_2$, and is the operator norm of $A$ mapping $\R^n$ to $\R^m$ in the Euclidean norm.  We consider families of submultiplicative, unitarily invariant norms $\|\cdot\|:\bigcup_{m,n=1}^\infty\K^{m\times n}\to\R$ which are \textit{normalized} ($\|x\|=\|x\|_2$ for any vector $x$ considered as a matrix) and \textit{uniformly generated} ($\|A\|$ can be written as $\phi(\sigma_1(A),\dots,\sigma_k(A))$ for some symmetric function $\phi$).  The canonical examples of such families of norms are the Schatten $p$--norms $(1\leq p\leq\infty)$ given by
$\|A\|_{S_p}:=\|(\sigma_1(A),\dots,\sigma_k(A))\|_{\ell_p}$.  Unfortunately, while $\|\cdot\|_2$ is a thoroughly reasonable notation for the spectral norm, it is actually the Schatten $\infty$--norm.  The Frobenius norm is the Schatten $2$--norm but is denoted $\|\cdot\|_F$, and the Nuclear norm is the Schatten $1$--norm, but is typically denoted $\|\cdot\|_*$;  unless we need to specify a specific choice or norm, we will simply use the symbol $\|\cdot\|$ to denote an arbitrary submultiplicative, unitarily invariant, normalized, uniformly generated norm. Note that $\|\cdot\|_2\leq\|\cdot\|$ for any such norm, and also that $\|AB\|\leq \|A\|_2\|B\|$.  

Finally, we will use $\mathcal{N}(A)$ and $\mathcal{R}(A)$ to denote the nullspace and range of $A$, respectively.

\subsection{Layout}

The rest of the paper consists of a discussion of CUR decompositions in Section \ref{SEC:CUR}, {\color{black}the setup for our perturbation analysis in Section \ref{SEC:Perturbation}, the main results and comparison with other facets of the literature in Section \ref{SEC:Main}, and refined estimates for maximal volume column and row submatrices in Section \ref{SEC:MaxVol}. Intermediate lemmas and proofs appear in Sections \ref{SEC:CURLemmas}--\ref{SEC:Proofs} and a discussion of rank-enforcement in CUR approximations is in Section \ref{SEC:Rank}. We end with some brief numerical experiments in Section \ref{SEC:Experiments} and comments in Section \ref{SEC:Conclusion}.  A supplementary proof and a table summarizing our error bounds are provided in Appendices \ref{APP:ProofPROP:NormTerms} and \ref{APP:Table}.}

\section{CUR Decompositions and Approximations}\label{SEC:CUR}

CUR approximations are low-rank approximations formed by selecting certain column and row submatrices of a given matrix, and then putting them together in some fashion. If $C$ and $R$ are such submatrices of $A$, then a \textit{CUR approximation} of $A$ is a product of the form {\color{black}$A\approx CU^\dagger R$}, where $C=A(:,J)$ for some $J\subset[n]$, $R=A(I,:)$ for some $I\subset[m]$, and $U=A(I,J)$. {\color{black} The middle matrix is chosen in various ways, but we will exclusively use $U$ to denote $A(I,J)$  here, and write other variants explicitly.}

For general $A$, there is a closed form for the best choice of $U$ for Frobenius norm error in the following sense.
\begin{proposition}[\cite{stewart_minimizer}]\label{PROP:CAR}
Let $A\in\K^{m\times n}$ and $C$ and $R$ be column and row submatrices of $A$, respectively.  Then the following holds:
\[ \underset{X}{\text{argmin}}\;\|A-CXR\|_F = C^\dagger AR^\dagger.\]
\end{proposition}
The approximation $A\approx CC^\dagger AR^\dagger R$ corresponds to projecting $A$ onto the span of the given columns and rows, which is a natural candidate for a good approximation (though interestingly Proposition \ref{PROP:CAR} does not hold for other norms, {\color{black}see \cite[Example 4.4]{HH2019}}).  The quality of a CUR approximation for matrices of full rank has been considered in many works in the theoretical Computer Science literature, e.g., \cite{BoutsidisOptimalCUR,DKMIII,DM05,DMM08,DMPNAS,SorensenDEIMCUR,WS_2017}.  Most of these works focus on randomly sampling columns and rows to form the approximation; however, these works consider many different choices for the middle matrix {\color{black}$U$} in the CUR approximation.  Nonetheless, there are deterministic methods of selecting columns given in \cite{SorensenDEIMCUR,VoroninMartinsson}, the latter of which first uses a fast QR factorization of $A$ and subsequently implicitly forms the CUR approximation.

In the event that $A$ is actually low rank, a characterization of exact \textit{CUR decompositions} was given by the authors in \cite{HH2019}, which we restate here for the reader's convenience.

\begin{theorem}[{\cite[Theorem 5.5]{HH2019}}]\label{THM:Characterization}
Let $A\in\K^{m\times n}$, $I\subset[m]$, and $J\subset[n]$.  Let $C=A(:,J)$, $R=A(I,:)$, and $U=A(I,J)$.  Then the following are equivalent:
\begin{enumerate}[(i)]
    \item $\rank(U)=\rank(A)$
    \item $A=CU^\dagger R$
    \item $A = CC^\dagger AR^\dagger R$
    \item $A^\dagger = R^\dagger UC^\dagger$
    \item $\rank(C)=\rank(R)=\rank(A)$.
\end{enumerate}
Moreover, if any of the equivalent conditions above hold, then $U^\dagger = C^\dagger AR^\dagger$.
\end{theorem}

Note that this theorem suggests at least two natural CUR \textit{approximations} to a general matrix $A$, namely $A\approx CC^\dagger AR^\dagger R$, and $A\approx CU^\dagger R$.  We will discuss both variants and several rank truncations in the sequel.

\section{Perturbations of CUR Approximations}\label{SEC:Perturbation}

We now turn to a perturbation analysis suggested by the CUR approximations described above.  Our primary task will be to consider matrices of the form \[\widetilde{A} = A+E,\] where $A$ has low rank $k<\min\{m,n\}$, and $E$ is an arbitrary (typically full-rank) noise matrix. We ask the question: if we choose column and row submatrices of $\widetilde{A}$, how do CUR \textit{approximations} of $\widetilde{A}$ of the forms suggested by Theorem \ref{THM:Characterization} relate to CUR \textit{decompositions} of $A$?  

To set some notation, we consider $\widetilde{C}=\widetilde{A}(:,J)$, $\widetilde{R}=\widetilde{A}(I,:)$, and $\widetilde{U}=\widetilde{A}(I,J)$ for some index sets $I$ and $J$, and we write 
\begin{equation}\label{EQ:tildes}
\widetilde{C} =  C+E(:,J),\quad \widetilde{R} = R+E(I,:),\quad \widetilde{U} = U+E(I,J),
\end{equation}
where $C := A(:,J)$, $R:=A(I,:)$ and $U:=A(I,J)$.
Thus if we choose columns and rows, $\widetilde{C}$ and $\widetilde{R}$ of $\widetilde{A}$, we seek to determine how approximation of $\widetilde{A}$ by $\widetilde{C}$ and $\widetilde{R}$ compares to the underlying approximation of the low rank matrix $A$ by its columns and rows, $C$ and $R$.  

For experimentation in the sequel we will consider $E$ to be a random matrix drawn from a certain distribution, but here we do not make any assumption on its entries.  We are principally interested in the case that $E$ is ``small" in a suitable sense, and so the observed matrix $\widetilde{A}$ is really a small perturbation of the low rank matrix $A$.  To this end, most of our analysis will contain upper bounds on a CUR approximation of $\widetilde{A}$ in terms of a norm of the noise $E$.

Note that we are interested in recovering the low-rank matrix $A$, but the approximations suggested above ($\widetilde{A}\approx\widetilde{C}\widetilde{U}^\dagger\widetilde{R}$ and $\widetilde{A}\approx\widetilde{C}\widetilde{C}^\dagger\widetilde{A}\widetilde{R}^\dagger\widetilde{R}$) are not necessarily low-rank.  Indeed, both approximations will typically have rank $\min\{|I|,|J|\}$, which could be larger than $k$ in general.  Therefore, we also consider various ways of enforcing the rank in the case that it is known or well-estimated.  Unfortunately, there is no canonical way to do this as we will demonstrate in Section \ref{SEC:Rank}.  Our perturbation estimates will analyze the following approximation errors:

\begin{itemize}
    \item $\|A-\widetilde{C}\widetilde{C}^\dagger\widetilde{A}\widetilde{R}^\dagger\widetilde{R}\|$
    \medskip
    \item $\|A-\widetilde{C}\widetilde{U}^\dagger\widetilde{R}\|$
    \medskip
    \item $\|A-\widetilde{C}[\widetilde{U}]_\tau^\dagger\widetilde{R}\|$
    \medskip
    \item $\|A-\widetilde{C}\widetilde{U}_k^\dagger\widetilde{R}\|$
    \medskip
    \item $\|A-\widetilde{C}_k\widetilde{C}_k^\dagger\widetilde{A}\widetilde{R}_k^\dagger\widetilde{R}_k\|$.
\end{itemize}
In our discussion in Section \ref{SEC:Rank}, we will also discuss the approximation $A\approx (\widetilde{C}\widetilde{U}^\dagger\widetilde{R})_k$.

For ease of notation, we will use the conventions that $E_I:=E(I,:)$, $E_J:=E(:,J)$, and $E_{I,J}:=E(I,J)$; since $I$ and $J$ are always reserved for subsets of the rows and columns, respectively, we trust this will not cause confusion. 

\subsection{Assumptions}

To make the statement of results more simple, we will always make the following assumptions throughout the rest of this paper.  $\widetilde{A} = A+E$ will be in $\K^{m\times n}$ with $\rank(A)=k$, and $C,U,R,\widetilde{C},\widetilde{U},\widetilde{R},E_I,E_J,E_{I,J}$ will be as in \eqref{EQ:tildes} with $I\subset[m]$ and $J\subset[n]$ being the row and column index sets, respectively.  We will always assume that $\rank(C)=\rank(U)=\rank(R)=k$, and that $\|\cdot\|$ is a normalized, uniformly generated, unitarily invariant, submultiplicative norm.  Given this assumption on the ranks, Proposition \ref{PROP:NormTerms} is valid and will be utilized frequently. 

{
\begin{remark}\label{REM:2norm}
For simplicity of reading, we state all bounds in the sequel for arbitrary norms satisfying the above assumptions; in particular, we use the pessimistic inequality $\|AB\|\leq \|A\|\|B\|$.  But we note that at any stage, we can use the fact that $\|AB\|\leq\|A\|_2\|B\|$, which gives a better bound.  In some instances, we will highlight how using the latter affects the right-hand sides of the given inequalities.
\end{remark}
}

{\color{black}\section{Main Results}\label{SEC:Main}
Now let us state the main results which are proven in the sequel.  In the bounds stated below, if $C=A(:,J)$ and $R=A(I,:)$ and $A=W_k\Sigma_kV_k^*$, we will often state error bounds in terms of norms of pseudoinverses of the corresponding row submatrices of the singular vectors $W_k$ and $V_k$.  To that end, we define $W_{k,I}:=W_k(I,:)$ and $V_{k,J}:=V_k(J,:)$.
}

\subsection{Perturbation Estimates for CUR Approximations With No Rank Enforcement} 

To begin, let us consider the CUR approximation suggested by the two exact decompositions of Theorem \ref{THM:Characterization}. 

\subsubsection{Projection Based Approximation: $A\approx \widetilde{C}\widetilde{C}^\dagger \widetilde{A}\widetilde{R}^\dagger\widetilde{R}$}\label{SEC:CCARR}

We begin our perturbation analysis by considering the approximation suggested by Theorem \ref{THM:Characterization}($iii$).  Our main result is the following.

\begin{theorem}\label{THM:ErrFApp}
The following holds:
\[ \|A-\widetilde{C}\widetilde{C}^\dagger \widetilde{A}\widetilde{R}^\dagger \widetilde{R}\|\leq \|E_I\|\|AR^\dagger\|+\|E_J\|\|C^\dagger A\|+3\|E\|.\]
Hence,
\[\|A-\widetilde{C}\widetilde{C}^\dagger \widetilde{A}\widetilde{R}^\dagger \widetilde{R}\|\leq \|E\|(\|W_{k,I}^\dagger\|+\|V_{k,J}^\dagger\|+3).\]
\end{theorem}

{\color{black}Theorem \ref{THM:ErrFApp} shows that the error in the projection-based CUR method is controlled by the pseudoinverses of the submatrices of the orthogonal singular vectors, and is linear in the norm of the noise $E$.  This bound is comparable to that of Sorensen and Embree \cite{SorensenDEIMCUR} which is of the form $\|\widetilde{A}-\widetilde{C}\widetilde{C}^\dagger\widetilde{A}\widetilde{R}^\dagger\widetilde{R}\|\leq\sigma_{k+1}(\widetilde{A})(\|W_{k,I}^{-1}\|+\|V_{k,J}^{-1}\|)$.  These bounds are not directly comparable in the general case because their proof requires $|I|=|J|=k$ (also, their result is stated only for the DEIM method of selecting $I$ and $J$ and only for the spectral norm, but holds more generally).  If we restrict to this case and take $E=\widetilde{A}-\widetilde{A}_k$, then our bound contains an extra $2\sigma_{k+1}(\widetilde{A})$ compared to theirs.  So we give up a small constant in our error bound in exchange for more flexible choices of column and row indices. These bounds also illustrate why the more successful approximation results for CUR approximations use the singular vectors of the input matrix to select columns; doing so can substantially decrease the norms of the matrices $W_{k,I}^\dagger$ and $V_{k,J}^\dagger$ appearing above. 

Note also that generic bounds for pseudoinverses of submatrices of truncated singular vectors can be hard to come by, but under additional incoherence assumptions on the matrix $\widetilde{A}$, one can give probabilistic bounds on the norms of $W_{k,I}$ and $V_{k,J}$ when $I$ and $J$ are sampled uniformly at random as was done in \cite{tropp2011improved} (and applied to CUR in \cite{DemanetWu}); however, we do not explore this here.
}

\subsubsection{Non-projection Based Approximation: $A\approx\widetilde{C}[\widetilde{U}]_\tau^\dagger\widetilde{R}$}\label{SEC:A-CUR}

Now we turn to considering the approximation suggested by Theorem \ref{THM:Characterization}($ii$).  We formulate our approximation in a slightly more general form by thresholding the singular values of $\widetilde{U}$ by a fixed parameter $\tau\geq0$.  Of course provided $0\leq\tau\leq\sigma_r(\widetilde{U})$ where $\rank(\widetilde{U})=r$, we have $\widetilde{C}[\widetilde{U}]_\tau^\dagger\widetilde{R} = \widetilde{C}\widetilde{U}^\dagger\widetilde{R}$, and hence this framework encompasses the case that no thresholding is actually done (recall the definition of $[\widetilde{U}]_\tau$ from Section \ref{SEC:Notation}).

This approximation scheme was studied by Osinsky et al.~\cite{Osinsky2018} and previously by Goreinov et al.~\cite{Goreinov2,Goreinov,Goreinov3}, and we recover similar perturbation results to those in the former, but by a different proof method, which we provide in full. The reason for including our analysis is that it gives some more qualitative estimates, and additionally we get slightly better error bounds since they are in terms of submatrices of the noise $E$.  Moreover, our bounds hold for arbitrary norms satisfying the conditions above (e.g., for all Schatten $p$--norms), which is a strengthening of the spectral and Frobenius norm guarantees of the aforementioned works.  Additionally, our estimation techniques are amenable to performing a novel analysis of different ways of enforcing the rank in the CUR approximation, which is done in the sequel.

\begin{theorem}\label{THM:CURTilde}
Given $\tau\geq0$, the following holds:
\begin{multline*}
    \|A-\widetilde{C}[\widetilde{U}]_\tau^\dagger\widetilde{R}\| \leq  \|W_{k,I}^\dagger\|\|E_I\|+\|V_{k,J}^\dagger\|\|E_J\|+\|W_{k,I}^\dagger\|\|V_{k,J}^\dagger\|(2\|E_{I,J}\|+\|[\widetilde{U}]_{\tau}-U\|)\\
     +  \|[\widetilde{U}]_\tau^\dagger\|\left[\left(\|W_{k,I}^\dagger\|\|E_I\|+\|V_{k,J}^\dagger\|\|E_J\|+\|W_{k,I}^\dagger\|\|V_{k,J}^\dagger\|\|E_{I,J}\|\right)\|E_{I,J}\|+\|E_I\|\|E_J\|  \right].
\end{multline*}
\end{theorem}

\begin{corollary}\label{COR:CURTilde}
Setting $\tau=0$, we have
\begin{multline*}
    \|A-\widetilde{C}\widetilde{U}^\dagger\widetilde{R}\| \leq  \left(\|W_{k,I}^\dagger\|+\|V_{k,J}^\dagger\|+3\|W_{k,I}^\dagger\|\|V_{k,J}^\dagger\|\right)\|E\|\\
     +  \|\widetilde{U}^\dagger\|\left(\|W_{k,I}^\dagger\|+\|V_{k,J}^\dagger\|+\|W_{k,I}^\dagger\|\|V_{k,J}^\dagger\|+1\right)\|E\|^2.
\end{multline*}
\end{corollary}

{
\begin{remark}
If the tighter bound suggested in Remark \ref{REM:2norm} is used, then the conclusion of Theorem \ref{THM:CURTilde} becomes
\begin{multline*}
    \|A-\widetilde{C}[\widetilde{U}]_\tau^\dagger\widetilde{R}\| \leq  \|W_{k,I}^\dagger\|_2\|E_I\|+\|V_{k,J}^\dagger\|_2\|E_J\|+\|W_{k,I}^\dagger\|_2\|V_{k,J}^\dagger\|_2(2\|E_{I,J}\|+\|[\widetilde{U}]_{\tau}-U\|)\\
     +  \|[\widetilde{U}]_\tau^\dagger\|_2\left[\left(\|W_{k,I}^\dagger\|_2\|E_I\|+\|V_{k,J}^\dagger\|_2\|E_J\|+\|W_{k,I}^\dagger\|_2\|V_{k,J}^\dagger\|_2\|E_{I,J}\|\right)\|E_{I,J}\|+\|E_I\|\|E_J\|  \right].
\end{multline*}
\end{remark}
}

\begin{remark}
The bounds above may be simplified in a couple of ways {\color{black}  for $\tau>0$}.  First, note that $\|[\widetilde{U}]_\tau^\dagger\|_2\leq\tau^{-1}$.  In some previous works, $\tau$ is chosen to offset the other norm terms to demonstrate the existence of nice upper bounds (e.g., taking $\tau=\|E\|$ or to be related to the inverse of the product $\|W_{k,I}^\dagger\|_2\|V_{k,J}^\dagger\|_2$).  This is not practical; however, we will mention in Section \ref{SEC:Comparison} how our results recover previous analyses in the literature in this direction.  Additionally, if the norm on the left-hand side is the spectral norm, we may estimate $\|[\widetilde{U}]_\tau-U\|_2\leq\tau+\|E_{I,J}\|_2$ by adding and subtracting $\widetilde{U}$ and applying the triangle inequality.
\end{remark}

Note that the approximation bound in Corollary \ref{COR:CURTilde} depends on $\widetilde{U}$ and hence must be considered preliminary, as this could be arbitrarily large.  Without additional assumptions, not much more may be said, but in Section \ref{SEC:MaxVol}, we analyze how one may improve the estimates herein by choosing maximal volume submatrices, and point the reader to some existing algorithms for doing so.

\subsection{Perturbation Estimates for Rank $k$ CUR Approximations} If $\rank(A)=k$ is known in advance, we are interested in enforcing this rank in any CUR approximation of $A$.  In this section, we consider two variants of rank enforcement.  Further discussion of the merits and drawbacks of both are found in Section \ref{SEC:Rank}.

\subsubsection{Enforcing the rank on $\widetilde{U}$} \label{SEC:CUkR}

If more than $k$ columns or rows of $\widetilde{A}$ are chosen, then the rank of $\widetilde{U}$ is typically larger than $k$.  Therefore, $\widetilde{C}\widetilde{U}^\dagger\widetilde{R}$ is an approximation of $A$ which has strictly larger rank.  It is natural to consider then what happens if the target rank is enforced.  There are many ways to enforce the rank, one of which that has been utilized for some time is to do so on the matrix $\widetilde{U}$.  If $\widetilde{A}$ is symmetric, positive semi-definite, then this rank-enforcement strategy is known to be deficient, but for generic matrices this is not so.  For further discussion on these matters, consult Section \ref{SEC:Rank}.  {\color{black}Note that below, $\widetilde{U}_k$ is the best rank $k$ approximation of $\widetilde{U}$, and $\widetilde{U}_k^\dagger$ is its Moore--Penrose pseudoinverse.}

\begin{theorem}\label{THM:PB}
{\color{black}Let $\mu\in[1,3]$ be the quantity given by Theorem \ref{THM:Stewart} ($\mu$ depends on the norm $\|\cdot\|$ chosen). Provided $\sigma_k(U)>2\mu\|E_{I,J}\|$, the following holds:}
    \begin{multline*}\label{IEQTHm4.10}
    \|A-\widetilde{C}\widetilde{U}_k^\dagger\widetilde{R}\| \leq  \left(\|W_{k,I}^\dagger\|\|E_I\|+\|V_{k,J}^\dagger\|\|E_J\|+4\|W_{k,I}^\dagger\|\|V_{k,J}^\dagger\|\|E_{I,J}\|\right)
     +  \frac{\|U^\dagger\|}{1-2\mu\|U^\dagger\|_2\|E_{I,J}\|}\times\\
     \left(\|W_{k,I}^\dagger\|\|E_I\|\|E_{I,J}\|+\|V_{k,J}^\dagger\|\|E_J\|\|E_{I,J}\|+\|W_{k,I}^\dagger\|\|V_{k,J}^\dagger\|\|E_{I,J}\|^2+\|E_I\|\|E_J\|\right).
\end{multline*}
\end{theorem}

\begin{remark}\label{REM:CUkR}
Note that all terms in the second line in the bound of Theorem \ref{THM:PB} are second order in the noise, whereas the first three terms are first order. In particular, if $\sigma_k(U)>4\mu\|E\|$, then  
\[
    \|A-\widetilde{C}\widetilde{U}_k^\dagger\widetilde{R}\| \leq  \left((1+\frac{1}{2\mu})(\|W_{k,I}^\dagger\|+\|V_{k,J}^\dagger\|)+(4+\frac{1}{2\mu})\|W_{k,I}^\dagger\|\|V_{k,J}^\dagger\|+\frac{1}{2\mu}\right)\|E\|.
\]
\end{remark}

{
\begin{remark}
Since $U^\dagger = W_{k,I}^\dagger \Sigma_k^\dagger V_{k,J}^\dagger$ and $\|\Sigma_k^\dagger\|=\|A^\dagger\|$, we may replace the fractional term in Theorem \ref{THM:PB} with
\[ \dfrac{\|W_{k,I}^\dagger\|\|V_{k,J}^\dagger\|\|A^\dagger\|}{1-2\mu\|W_{k,I}^\dagger\|_2\|V_{k,J}^\dagger\|_2\|A^\dagger\|_2\|E_{I,J}\|}\]
thus giving a bound independent of the chosen $U$.  Indeed, this means that the error bounds in Theorem \ref{THM:PB} are of the form
\[ \|A-\tilde{C}\tilde{U}_k^\dagger\tilde{R}\|\leq\|A-CU^\dagger R\| + O(\|E\|) + O(\|A^\dagger\|\|E\|^2). \]
That is, the first order terms depend essentially only on the noise, whereas the second order terms have dependence on $\|A^\dagger\|$.  Do note that the assumptions in Theorem \ref{THM:PB} imply that $\|E\|\|A^\dagger\|\leq C_1$ for some universal constant $C_1$, {\color{black}so the estimate on the right-hand side is still $O(\|E\|)$}; on the other hand, it could be that this quantity is small in some instances, so we leave the expression as is to denote the second order dependence on the noise matrix.
\end{remark}
}

\subsubsection{Projection onto $k$--dimensional subspaces}\label{SEC:CkCk}

Next let us consider what happens if we enforce the rank on both $\widetilde{C}$ and $\widetilde{R}$ and then form the projection-based approximation from these.  This corresponds to finding the best $k$--dimensional subspace that approximates the span of the columns of $\widetilde{C}$ and projecting $\widetilde{A}$ onto this subspace.  Ideally, this should well-approximate projecting $\widetilde{A}$ onto the span of the columns of $C$ itself.

\begin{theorem}\label{THM:CkARk}
We have
\[ \|A-\widetilde{C}_k\widetilde{C}_k^\dagger \widetilde{A}\widetilde{R}_k^\dagger \widetilde{R}_k\|\leq 2\left(\|E_J\|\|W_{k,I}^\dagger\|+\|E_I\|\|V_{k,J}^\dagger\|+\frac32\|E\|\right). \]
Hence,
\[\|A-\widetilde{C}_k\widetilde{C}_k^\dagger \widetilde{A}\widetilde{R}_k^\dagger \widetilde{R}_k\|\leq 2\|E\|\left(\|W_{k,I}^\dagger\|+\|V_{k,J}^\dagger\|+\frac32\right).\]
\end{theorem}

{\color{black}The CUR approximation of the form $A\approx \widetilde{C}_k\widetilde{C}_k^\dagger \widetilde{A}\widetilde{R}_k^\dagger \widetilde{R}_k$ appears to be novel, although it is completely natural to consider.  The estimates here are modestly worse than in the case where no rank truncation is done (Theorem \ref{THM:ErrFApp}), but still contain the same terms on the right-hand side.  We suggest that this method ought to be explored further in the future as evidenced by the numerical experiments in Section \ref{SEC:Experiments}. 

One of the common themes in utilizing CUR approximations is that while in the low-rank decomposition case (as in Theorem \ref{THM:Characterization}), one has $A=CU^\dagger R = CC^\dagger AR^\dagger R$, there is a tradeoff between using these in practice.  It is computationally simpler to compute $U^\dagger$, but as a low-rank approximation, $A\approx CU^\dagger R$ often exhibits poor performance.  On the other hand, $A\approx CC^\dagger AR^\dagger R$ typically yields a very good approximation (recall Proposition \ref{PROP:CAR}) but at the cost of being more costly to compute.  The approximation scheme proposed in Theorem \ref{THM:CkARk} gives an alternative which has good approximation power while having mildly higher complexity given that one must compute the truncated SVD of both $C$ and $R$.
}

\section{Refined Estimates: The Maximal Volume Case}\label{SEC:MaxVol}

One drawback of the above estimates is that some of the right-hand sides maintain dependencies on the choice of the submatrix $U$.  If one assumes that maximal volume submatrices of the left and right singular values are chosen, then one can use estimates from \cite{Osinsky2018} to give bounds on the corresponding spectral norms.  Recall that the volume of a matrix $B\in\K^{m\times n}$ is $\prod_{i=1}^{\min\{m,n\}}\sigma_i(B)$.  While finding the maximal volume submatrix of a given matrix is NP--hard, there are good approximation algorithms available, e.g. \cite{goreinov2010find, mikhalev2018rectangular,osinsky2018rectangular}.

\subsection{Properties of Maximal Volume Submatrices}

\begin{proposition}\label{PROP:MaxVol}
Suppose that $W_{k,I}$ and $V_{k,J}$ are the submatrices of $W_k$ and $V_k$ such that $W_{k,I}$ has maximal volume among all $|I|\times k $ submatrices of $W_k$ and $V_{k,J}$ is of maximal volume among all $|J|\times k$ submatrices of $V_k$. Then 

\begin{equation}\label{EQ:MaxVol}
\|W_{k,I}^\dagger\|_2 \leq \sqrt{1+\frac{k(m-|I|)}{|I|-k+1}},\quad\|V_{k,J}^\dagger\|_2\leq \sqrt{1+\frac{k(n-|J|)}{|J|-k+1}}.
\end{equation}
Moreover, if $\rank(U)=\rank(A)$, then 
\begin{equation}\label{EQ:MaxVol2} \|U^\dagger\|_2\leq \sqrt{1+\frac{k(m-|I|)}{|I|-k+1}} \sqrt{1+\frac{k(n-|J|)}{|J|-k+1}}\|A^\dagger\|_2.\end{equation}
\end{proposition}

Note that \eqref{EQ:MaxVol} appears in \cite{Osinsky2018}, and the moreover statement follows by Proposition \ref{PROP:CAR} and the assumption that $\rank(U)=\rank(A)$.  For ease of notation, since the upper bounds appearing in \eqref{EQ:MaxVol} are universal, we abbreviate the quantities there $t(k,m,|I|)$, and $t(k,n,|J|)$, respectively as in \cite{Osinsky2018}.  Regard also that Frobenius bounds are also provided in \cite{Osinsky2018}, where the upper bound is $\widetilde{t}(k,m,|I|) = \sqrt{k+\frac{k(m-|I|)}{|I|-k+1}}$.

To finish off our perturbation analysis, we will apply the conclusion of Proposition \ref{PROP:MaxVol} to the bounds on rank-enforcement CUR approximations discussed in the previous section.  We withhold the proofs as the following corollaries arise by simply applying \eqref{EQ:MaxVol} and \eqref{EQ:MaxVol2} to previous theorems.

As a preliminary remark, we note that choosing maximal volume submatrices of the singular vectors of $A$ automatically yields a valid CUR decomposition.
\begin{proposition}
Let $A\in\K^{m\times n}$ have rank $k$ and compact SVD $A=W_k\Sigma_kV_k^*$.  Suppose that $I\subset[m]$ and $J\subset[n]$ satisfy $|I|,|J|\geq k$ and $W_{k,I}$ and $V_{k,J}$ are the maximal volume submatrices of $W_k$ and $V_k$, respectively.  If $C=A(:,J)$, $U=A(I,J)$, and $R=A(:,J)$, then $A=CU^\dagger R$.
\end{proposition}
\begin{proof}
Notice that $\rank(W_k)=\rank(V_k)=k$.  There exist $I_1$, $J_1$ such that $\rank(W_{k,I_1})=k$ and $\rank(V_{k,J_1})=k$. Therefore, by the definition of the maximal volume submatrices, we must have $\rank(W_{k,I})=\rank(V_{k,J})=k$; indeed, recall that the volume of $W_{k,I} = \prod_{j=1}^k \sigma_j(W_{k,I})$, and this product is $0$ if $W_{k,I}$ has rank less than $k$ and hence cannot be of maximal volume since $W_{k,I_1}$ has nonzero volume. It follows that $\rank(C)=\rank(R)=k$, and hence $A=CU^\dagger R$ by Theorem \ref{THM:Characterization}.
\end{proof}

The following corollary arises from Theorems \ref{THM:PB} and \ref{THM:CkARk} and Remark \ref{REM:2norm}; {the condition on $\sigma_k(A)$ comes from combining \eqref{EQ:MaxVol2} with the condition on $\sigma_k(U)$ in Lemma \ref{LEM:PBTerms}.}
\begin{corollary}\label{COR:CURMaxVol}
Suppose that $A\in\K^{m\times n}$ has rank $k$ and compact  SVD $A=W_k\Sigma_kV_k^*$.  Suppose also that $W_{k,I}$ and $V_{k,J}$ are maximal volume submatrices of $W_k$ and $V_k$, respectively.  Then
\[\|A-\widetilde{C}_k\widetilde{C}_k^\dagger\widetilde{A}\widetilde{R}_k^\dagger\widetilde{R}_k\|\leq (2t(k,m,|I|)+2t(k,n,|J|)+3)\|E\|.\]
If additionally, $\sigma_k(A)\geq 4\mu t(k,m,|I|)t(k,n,|J|)\|E\|$, then 
\begin{multline*} \|A-\widetilde{C}\widetilde{U}_k^\dagger\widetilde{R}\|\leq  \left[\frac{1}{2\mu}+\left(1+\frac{1}{2\mu}\right)\left(t(k,m,|I|)+t(k,n,|J|)\right)\right.\\
\left.+\left(4+\frac{1}{2\mu}\right)t(k,m,|I|)t(k,n,|J|)\right]\|E\|. 
\end{multline*}
\end{corollary}
\begin{corollary}\label{COR:CURMaxVol1}
Suppose $A\in\K^{n\times n}$ and $|I|=|J|$.   Suppose also that $W_{k,I}$ and $V_{k,J}$ are maximal volume submatrices as in Proposition \ref{PROP:MaxVol}.  Abbreviate $t:=t(k,n,|I|)$.  Then\[\|A-\widetilde{C}_k\widetilde{C}_k^\dagger\widetilde{A}\widetilde{R}_k^\dagger\widetilde{R}_k\|\leq (4t+3)\|E\|.\]
If additionally, $\sigma_k(A)\geq 4\mu t^2\|E\|$, then 
\[ \|A-\widetilde{C}\widetilde{U}_k^\dagger\widetilde{R}\|\leq  \left(\frac{1}{2\mu}+\left(2+\frac{1}{\mu}\right)t+\left(4+\frac{1}{2\mu}\right)t^2\right)\|E\|.\]
\end{corollary}

\subsection{Comparison with Previous Results}\label{SEC:Comparison}

Osinsky and Zamarashkin \cite{Osinsky2018} provides several estimates of CUR approximations in which they assume that maximal volume submatrices are chosen.

Here are some of the theorems therein stated for comparison.  For simplicity of the statements, we focus on the case that $A$ is square, and exactly $k$ columns and rows are selected (i.e., $|I|=|J|=k$) and denote the factor $t$ as in Corollary \ref{COR:CURMaxVol1}.  The first result we highlight is the following.
\begin{theorem}[{\cite[Theorem 2]{Osinsky2018}}]\label{THM:Osinsky1}
There exist $I,J\subset[n]$ and $X$ such that
\[\|A-\widetilde{C}X\widetilde{R}\|_2 \leq 4t\|E\|_2. \]

\end{theorem}

This error bound is relatively good; however, the authors choose $X=[A(I,J)]_{\tau}^\dagger$ with $\tau=\frac{\|E\|}{t}$.  This choice is impractical for real matrices as one does not have access to $A(I,J)$ or $\|E\|$ in practice. For comparison, Theorem \ref{THM:ErrFApp} yields an upper bound of $(3+2t)\|E\|$. It is easily checked that when $t>3/2$, Theorem \ref{THM:ErrFApp} gives a better bound than Theorem \ref{THM:Osinsky1}, but nonetheless this approximation still suffers from needing to know the SVD of $A$.

The second result we highlight is the following.
\begin{theorem}[{\cite[Theorem 3]{Osinsky2018}}]
If $\|E\|_2\leq\eps$, and $I$ and $J$ yield the maximal volume submatrices of $W_{k}$ and $V_k$, respectively, then
\[\|\widetilde{A}-\widetilde{C}[\widetilde{U}]_\eps^\dagger\widetilde{R}\|_2 \leq (2 + 4t + 5t^2)\eps. \]
\end{theorem}

Our estimate in Theorem \ref{THM:CURTilde} gives the same error bound  as  this on the right-hand side if we choose $\tau=\eps$. However, our estimates obtained by directly enforcing the rank are novel and not directly considered in other works. Indeed, in certain cases, Remark \ref{REM:CUkR} gives a much better bound than that above: 
\begin{corollary}
    If $\|E\|\leq\eps$, $I$ and $J$ yield maximal volume submatrices of $W_k$ and $V_k$, respectively, and in addition $\sigma_k(U)>4\mu\eps$, then
    \[ \|A-\widetilde{C}\widetilde{U}_k^\dagger\widetilde{R}\|\leq\left(\frac12+3t+\frac92 t^2\right)\eps.\]
\end{corollary}

Note that bounds of a different flavor have recently been provided by Mikhalev and Oseledets \cite{mikhalev2018rectangular}.  Additionally, independent, concurrent work which has an analysis of $\widetilde{A}-\widetilde{C}\widetilde{U}^\dagger\widetilde{R}$ in a similar vein to that of Section \ref{SEC:A-CUR} was done by Pan et al.~\cite{pan2019cur}, though their subsequent focus is more algorithmic than the present work.   





\section{Useful Properties of CUR Decompositions}\label{SEC:CURLemmas}

Here we collect some useful properties of the submatrices involved in exact CUR decompositions.  The proofs of Lemma \ref{LEM:Projections} and Propositions \ref{PROP:U=RAC} and \ref{PROP:Udagger} may be found in \cite{HH2019}.  

\begin{lemma}\label{LEM:Projections}
Suppose that $A, C, U,$ and $R$ are as in Theorem \ref{THM:Characterization}, with $\rank(A)=\rank(U)$.  Then $\mathcal{N}(C)=\mathcal{N}(U)$, $\mathcal{N}(R^*)=\mathcal{N}(U^*)$, $\mathcal{N}(A)=\mathcal{N}(R)$, and $\mathcal{N}(A^*)=\mathcal{N}(C^*)$.  Moreover,
\[ C^\dagger C = U^\dagger U,\qquad  RR^\dagger=UU^\dagger,\]
\[ AA^\dagger = CC^\dagger, \quad \text{and} \quad A^\dagger A = R^\dagger R.\]
\end{lemma}

\begin{proposition}\label{PROP:U=RAC}
	Suppose that $A$, $C$, $U$, and $R$  are as in Theorem \ref{THM:Characterization} (but without any assumption on the rank of $U$). Then \[U=RA^{\dagger}C.\]
\end{proposition}

\begin{proposition}\label{PROP:Udagger}
Suppose that $A, C, U,$ and $R$ satisfy the conditions of Theorem \ref{THM:Characterization}.  Then \[ U^\dagger=C^\dagger A R^\dagger.\]
\end{proposition}

Our final proposition will be useful in estimating some of the terms that arise in the subsequent analysis.
\begin{proposition}\label{PROP:NormTerms}
Suppose that $A, C, U$, and $R$ are as in Theorem \ref{THM:Characterization} {\color{black}(with selected row and column indices being $I$ and $J$, respectively)} such that $A=CU^\dagger R$, and suppose that $\rank(A)=k$.  Let $A=W_k\Sigma_k V_k^*$ be the compact SVD of $A$.  Then for any unitarily invariant norm $\|\cdot\|$ on $\K^{m\times n}$, we have
\[ \|CU^\dagger\|= \|W_{k,I}^\dagger\|,\quad \text{and} \quad \|U^\dagger R\|= \|V_{k,J}^\dagger\|,\]
where $W_{k,I}:=W_k(I,:)$ and $V_{k,J}:=V_k(J,:)$.
\end{proposition}
\begin{proof}
See Appendix \ref{APP:ProofPROP:NormTerms}.
\end{proof}

Unfortunately, it is often difficult to say much about the norms of pseudoinverses of submatrices of the compact SVD of a matrix; however, we will give some indications later of some universal bounds that can be used in certain cases. 

\section{Preliminaries from Matrix Perturbation Theory}\label{SEC:MatrixPerturbation}
Here we collect some useful facts from perturbation theory. The first is due to Weyl:

\begin{theorem}\cite[Corollary 8.6.2.]{GolubVanLoan} \label{THMHoffmanWielandt}
	If $B,E\in\K^{m\times n}$ and $\widetilde{B}=B+E$, then for $1\leq j\leq \min\{m,n\}$,
	\begin{equation}
	\left|\sigma_j(B)-\sigma_j(\widetilde{B})\right|\leq \sigma_1(E)=\|E\|_2.
	\end{equation}
\end{theorem}

Note that Theorem \ref{THMHoffmanWielandt} holds in greater generality and is due to Mirsky \cite{Mirsky}.  Therein, it was shown that for any normalized, uniformly generated, unitarily invariant norm $\|\cdot\|$, \begin{equation}\label{EQ:Mirsky}\|\textnormal{diag}(\sigma_1(B)-\sigma_1(\widetilde{B}), \sigma_2(B)-\sigma_2(\widetilde{B}),\cdots)\|\leq\|E\|.\end{equation}  

The following Theorem of Stewart provides an estimate for how large the difference of pseudoinverses can be.  

\begin{theorem}\cite[Theorems 3.1--3.4]{Stewart_1977}\label{THM:Stewart}
	Let $\|\cdot\|$ be any normalized, uniformly generated, unitarily invariant norm on $\K^{m\times n}$.  For any $B,E\in\K^{m\times n}$ with $\widetilde{B}=B+E$, if $\rank(\widetilde{B})=\rank(B)$, then
	\[
	\|B^{\dagger}-\widetilde{B}^{\dagger} \|\leq\mu \|\widetilde{B}^{\dagger}\|_2\|B^{\dagger}\|_2\|E\|,
	\]
	where $1\leq\mu\leq 3$ is a constant depending only on the norm.
    
    If $\rank(\widetilde{B})\neq\rank(B)$, then
    \[ \|B^\dagger-\widetilde{B}^\dagger\|\leq\mu\max\{\|\widetilde{B}^\dagger\|_2^2,\|B^\dagger\|_2^2\}\|E\| \text{ and } \frac{1}{\|E\|_2}\leq\|B^\dagger-\widetilde{B}^\dagger\|_2. \]
\end{theorem}

The precise value of $\mu$ depends on the norm used and the relation of the rank of the matrices to their size; in particular, $\mu=3$ for an arbitrary norm satisfying the hypotheses in Section \ref{SEC:Notation}, whereas $\mu=\sqrt{2}$ for the Frobenius norm, and $\mu=\frac{1+\sqrt{5}}{2}$ (the Golden Ratio) for the spectral norm.

The preceding theorems yield the following immediate corollary.

\begin{corollary}\label{COR:AdaggerBounds}
	With the assumptions of Theorem \ref{THM:Stewart}, if $\widetilde{B}=B+E$ and $\rank(\widetilde{B})=\rank(B)=k$, then
	\[|\|B^{\dagger}\|-\|\widetilde{B}^{\dagger}\||\leq \mu\|B^\dagger\|_2\|\widetilde{B}^{\dagger}\|_2\|E\|.\]
Moreover, if $\sigma_k(B)>\mu\|E\|$, then
	\[\frac{\|B^\dagger\|}{1+\mu\|B^\dagger\|_2\|E\|}\leq\|\widetilde{B}^{\dagger}\|\leq \frac{\|B^\dagger\|}{1-\mu\|B^\dagger\|_2\|E\|}.\]
\end{corollary}

Regard that from the representation of $B^\dagger$ in terms of the SVD of $B$ mentioned in Section \ref{SEC:Notation}, we have $\|B^\dagger\|_2=1/\sigma_{\min}(B)$, where $\sigma_{\min}(B)$ is the smallest nonzero singular value of $B$; this is sometimes how the inequalities in Corollary \ref{COR:AdaggerBounds} are written.

\section{Proofs}\label{SEC:Proofs}

\subsection{Proofs from Section \ref{SEC:CCARR}}
Before proving Theorem \ref{THM:ErrFApp}, we need the following lemma.

\begin{lemma}\label{LEM:CCA1}
The following hold: 
\begin{align*}
\|A-\widetilde{C}\widetilde{C}^\dagger \widetilde{A}\| &\leq \|E_J\|\|C^\dagger A\|+\|E\|,\\
 \|A-\widetilde{A}\widetilde{R}^\dagger\widetilde{R}\| &\leq \|E_I\|\|AR^\dagger\|+\|E\|.
\end{align*}
\end{lemma}
\begin{proof}
First, notice that
\begin{eqnarray*}
\|(I-\widetilde{C}\widetilde{C}^\dagger)C\|&=&\|(I-\widetilde{C}\widetilde{C}^\dagger)\widetilde{C}-(I-\widetilde{C}\widetilde{C}^\dagger)E_J\|\\
&\leq&\|(I-\widetilde{C}\widetilde{C}^\dagger)\widetilde{C}\|+\|(I-\widetilde{C}\widetilde{C}^\dagger)E_J\|\\
&\leq&\|E_J\|.
\end{eqnarray*}
The final inequality arises because the first norm term is 0 by identity of the Moore--Penrose pseudoinverse and $\|I-\widetilde{C}\widetilde{C}^\dagger\|_2\leq1$ as this is an orthogonal projection operator.  Now since $\rank(C)=\rank(A)=k$, we have $A=CC^\dagger A$; therefore,
\begin{eqnarray*}
\| A-\widetilde{C}\widetilde{C}^\dagger \widetilde{A}\|&\leq&\|(I-\widetilde{C}\widetilde{C}^\dagger) A\|+\|E\|\\
&=&\|(I-\widetilde{C}\widetilde{C}^\dagger)CC^\dagger A\|+\|E\|\\
&\leq&\|E_J\|\|C^\dagger A\|+\|E\|.
\end{eqnarray*}
The second inequality follows by mimicking the above argument.
\end{proof}

\begin{proof}[Proof of Theorem \ref{THM:ErrFApp}]
First note that \[\|A-\widetilde{C}\widetilde{C}^\dagger\widetilde{A}\widetilde{R}^\dagger\widetilde{R}\|\leq \|A-\widetilde{C}\widetilde{C}^\dagger\widetilde{A}\| + \|\widetilde{A}-\widetilde{A}\widetilde{R}^\dagger\widetilde{R}\|\]
by the triangle inequality and the fact that $\|\widetilde{C}\widetilde{C}^\dagger\|_2\leq1$.  The proof is completed by first noting that the second term above satisfies $\|\widetilde{A}(I-\widetilde{R}^\dagger\widetilde{R})\|\leq \|E\|+\|A(I-\widetilde{R}^\dagger\widetilde{R})\|$ since $I-\widetilde{R}^\dagger\widetilde{R}$ is a projection, and then applying the inequalities of Lemma \ref{LEM:CCA1}.  The second stated inequality follows directly by Proposition \ref{PROP:NormTerms} and the fact that the norms of submatrices of $E$ are at most $\|E\|$.
\end{proof}

\subsection{Proofs for Section \ref{SEC:A-CUR}}

\begin{proposition}\label{PROP:Perturbation1}
Let $\tau\geq0$ be fixed; then the following holds:
\begin{multline*} \|A-\widetilde{C}[\widetilde{U}]_{\tau}^{\dagger}\widetilde{R}\|\leq \|C[\widetilde{U}]_{\tau}^\dagger\|\|E_I\|+\|[\widetilde{U}]_{\tau}^\dagger \widetilde{R}\|\|E_J\| \\ +\|W_{k,I}^\dagger\|\|V_{k,J}^\dagger\|\left[2\|E_{I,J}\|+\|[\widetilde{U}]_\tau-U\|+\|[\widetilde{U}]_\tau^\dagger\|\|E_{I,J}\|^2\right].\end{multline*}
\end{proposition}
\begin{proof}
Begin with the fact that \[ \|A-\widetilde{C}[\widetilde{U}]_{\tau}^\dagger \widetilde{R}\|\leq \|A-CU^\dagger R\|+\|CU^\dagger R-\widetilde{C}[\widetilde{U}]_{\tau}^\dagger \widetilde{R}\|,\]
and notice that the first term is 0 by the assumption on $U$.  Then we have
	\begin{eqnarray*}
	\|CU^\dagger R-\widetilde{C}[\widetilde{U}]_{\tau}^\dagger \widetilde{R}\|&\leq&\|CU^\dagger R-C[\widetilde{U}]_{\tau}^\dagger R\|+\|C[\widetilde{U}]_{\tau}^\dagger R -C[\widetilde{U}]_{\tau}^\dagger \widetilde{R}\|+\|C[\widetilde{U}]_{\tau}^\dagger \widetilde{R}-\widetilde{C}[\widetilde{U}]_{\tau}^\dagger \widetilde{R}\|\\
	&\leq&\|CU^\dagger R-C[\widetilde{U}]_{\tau}^\dagger R\|+\|C[\widetilde{U}]_{\tau}^\dagger\|_2\| R- \widetilde{R}\|+\|C - \widetilde{C}\|\|[\widetilde{U}]_{\tau}^\dagger \widetilde{R}\|_2 \\
    & = & \|CU^\dagger R-C[\widetilde{U}]_{\tau}^\dagger R\|+\|C[\widetilde{U}]_{\tau}^\dagger\|_2\|E_I\|+\|[\widetilde{U}]_{\tau}^\dagger \widetilde{R}\|_2\|E_J\|.
	\end{eqnarray*}

To estimate the first term above, note that Lemma \ref{LEM:Projections} implies that $C = CC^\dagger C = CU^\dagger U$, and likewise $R=RR^\dagger R = UU^\dagger$; {\color{black}with the additional fact that $\widetilde{U}[\widetilde{U}]_\tau^\dagger\widetilde{U} = [\widetilde{U}]_\tau$}, the following holds:
\begin{eqnarray}\label{EQCURDiff1}
	\|CU^\dagger R-C[\widetilde{U}]_{\tau}^\dagger R\|&=&\|CU^\dagger UU^\dagger R-CU^\dagger(\widetilde{U}-E_{I,J})[\widetilde{U}]_{\tau}^\dagger(\widetilde{U}-E_{I,J})U^\dagger R\| \nonumber\\
	&\leq&\|CU^\dagger(U-[\widetilde{U}]_{\tau})U^\dagger R\|+\|CU^\dagger\widetilde{U} [\widetilde{U}]_\tau^\dagger  E_{I,J} U^\dagger R\|+ \nonumber\\
	&&\|CU^\dagger E_{I,J}[\widetilde{U}]_\tau^\dagger \widetilde{U} U^\dagger R\|+\|CU^\dagger E_{I,J}[\widetilde{U}]_\tau^\dagger  E_{I,J} U^\dagger R\|.
	\end{eqnarray}
The first term in \eqref{EQCURDiff1} is evidently at most $\|CU^\dagger\|\|U^\dagger R\|\|[\widetilde U]_\tau-U\|$, whereas the second is majorized by the same quantity on account of the fact that $\widetilde{U}[\widetilde{U}]_\tau^\dagger$ is a projection.  Similarly, as $[\widetilde{U}]_\tau^\dagger\widetilde{U}$ is a projection, the third term in \eqref{EQCURDiff1} is at most $\|CU^\dagger\|\|U^\dagger R\|\| E_{I,J}\|$, while the final term is at most $\|CU^\dagger\|\|U^\dagger R\|\|[\widetilde{U}]_\tau^\dagger\|\|E_{I,J}\|^2.$  Putting these observations together, and combining \eqref{EQCURDiff1} with Proposition \ref{PROP:NormTerms} yields the following:

 \begin{equation}\label{EQ:CUR-CtildeUR1}
    \|CU^\dagger R-C[\widetilde{U}]_\tau^\dagger R\| \leq \|W_{k,I}^\dagger\|\|V_{k,J}^\dagger\|(2\|E_{I,J}\|+\|[\widetilde{U}]_\tau-U\|+\|[\widetilde{U}]_\tau^\dagger\|_2\|E_{I,J}\|^2).
	\end{equation}
	Combining the estimates of \eqref{EQCURDiff1} and \eqref{EQ:CUR-CtildeUR1} yields the desired conclusion.
\end{proof}

\begin{lemma}\label{LEM:PBTerms0}
If $\tau\geq0$, then the following hold:
\begin{enumerate}[(i)]
\item\label{ITEM:CUkBound} 
$\|C[\widetilde{U}]_{\tau}^\dagger\|\leq \|[\widetilde{U}]_\tau^\dagger\|\|E_{I,J}\|\|W_{k,I}^\dagger\|+\|W_{k,I}^\dagger\|$,

\medskip
\item\label{ITEM:UkRBound} 
$\|[\widetilde{U}]_\tau^\dagger\widetilde{R}\|\leq \|[\widetilde{U}]_\tau^\dagger\|\left(\|E_{I,J}\|\|V_{k,J}^\dagger\|+\|E_I\|\right)+\|V_{k,J}^\dagger\|.$
\end{enumerate}
\end{lemma}
\begin{proof}
 To see  $(\ref{ITEM:CUkBound})$, notice that $C=CU^\dagger U$ by Theorem \ref{THM:Characterization}, whence applying Proposition \ref{PROP:NormTerms} yields

\begin{eqnarray*}
\|C[\widetilde{U}]_\tau^\dagger\|&=&\|CU^\dagger U[\widetilde{U}]_\tau^\dagger\| \\
&\leq&\|CU^\dagger\|\|U[\widetilde{U}]_\tau^\dagger\|_2\\
&=&\|W_{k,I}^\dagger\|\|(\widetilde{U}-E_{I,J})[\widetilde{U}]_\tau^\dagger\|_2\\
&\leq& \|W_{k,I}^\dagger\|(\|\widetilde{U}[\widetilde{U}]_\tau^\dagger\|_2+\|E_{I,J}\|_2\|[\widetilde{U}]_\tau^\dagger\|_2)\\
&\leq&\|W_{k,I}^\dagger\|(1+\|E_{I,J}\|\|[\widetilde{U}]_\tau^\dagger\|).
\end{eqnarray*}

Similarly, we have \[  \|[\widetilde{U}]_\tau^\dagger R \|\leq \|V_{k,J}^\dagger\|(1+\|E_{I,J}\|\|[\widetilde{U}]_\tau^\dagger\|).\]
Thus to prove $(\ref{ITEM:UkRBound})$, note that
\begin{eqnarray*}
\|[\widetilde{U}]_\tau^\dagger \widetilde{R} \|&\leq&\|[\widetilde{U}]_\tau^\dagger R\|+\|[\widetilde{U}]_\tau^\dagger E_I\| \\
&\leq&\|V_{k,J}^\dagger\|(1+\|E_{I,J}\|\|[\widetilde{U}]_\tau^\dagger\|)+\|[\widetilde{U}]_\tau^\dagger\|\|E_I\|.
\end{eqnarray*}
\end{proof}

\begin{proof}[Proof of Theorem \ref{THM:CURTilde}]
Apply the conclusion of Lemma \ref{LEM:PBTerms0} to Proposition \ref{PROP:Perturbation1} and collect terms.
\end{proof}

\subsection{Proofs for Section \ref{SEC:CUkR}}
   By modifying the proof of Proposition \ref{PROP:Perturbation1} and Lemma \ref{LEM:PBTerms0}, we arrive at the following.      

\begin{proposition}\label{PROP:PB}
Let $\widetilde{U}_k$ be the best rank-$k$ approximation of $\widetilde{U}$.  Then 
\begin{multline*}
    \|A-\widetilde{C}\widetilde{U}_k^\dagger\widetilde{R}\| \leq  \|W_{k,I}^\dagger\|\|E_I\|+\|V_{k,J}^\dagger\|\|E_J\|+4\|W_{k,I}^\dagger\|\|V_{k,J}^\dagger\|\|E_{I,J}\|\\
     +  \|\widetilde{U}_k^\dagger\|\left[\left(\|W_{k,I}^\dagger\|\|E_I\|+\|V_{k,J}^\dagger\|\|E_J\|+\|W_{k,I}^\dagger\|\|V_{k,J}^\dagger\|\|E_{I,J}\|\right)\|E_{I,J}\|+\|E_I\|\|E_J\|  \right].
\end{multline*}
\end{proposition}

The presence of terms depending on $\widetilde{U}$ in the error bounds above are undesirable, so we now are tasked with estimating them.  Before stating the final bound, we estimate some of the terms specifically in the following lemma.

\begin{lemma}\label{LEM:PBTerms}
Let $\mu\in[1,3]$ be the quantity given by Theorem \ref{THM:Stewart} ($\mu$ depends on the norm $\|\cdot\|$ chosen). Provided $\sigma_k(U)>2\mu\|E_{I,J}\|$, the following estimate holds:
\[ \|\widetilde{U}_k^\dagger\|\leq\dfrac{\|U^\dagger\|}{1-2\mu\|U^\dagger\|_2\|E_{I,J}\|}.\]
\end{lemma}
\begin{proof}
Note that $\widetilde{U}_k = U + (\widetilde{U}_k-U)$, and notice that $\|U-\widetilde{U}_k\|\leq\|U-\widetilde{U}\|+\|\widetilde{U}-\widetilde{U}_k\|$, where the first term is equal to $\|E_{I,J}\|$ by definition, and the second satisfies $\|\widetilde{U}-\widetilde{U}_k\|\leq\|E_{I,J}\|$ by Mirsky's Theorem.  Hence $\|U-\widetilde{U}_k\|\leq2\|E_{I,J}\|$.  Using this estimate, we see that if $\sigma_k(U)>2\mu\|E_{I,J}\|\geq\mu\|\widetilde{U}_k-U\|$, then by Corollary \ref{COR:AdaggerBounds}, 
\[ \|\widetilde{U}_k^\dagger\|\leq \dfrac{\|U^\dagger\|}{1-\mu\|U^\dagger\|_2\|\widetilde{U}_k-U\|}\leq \dfrac{\|U^\dagger\|}{1-2\mu\|U^\dagger\|_2\|E_{I,J}\|},\] which is the desired conclusion. {\color{black} Note that the use of Corollary \ref{COR:AdaggerBounds} requires that $\rank(\widetilde{U})\geq k$, but this is implied by the condition relating $\sigma_k(U)$ and $\|E_{I,J}\|$.  Indeed, by Weyl's inequality and this assumption, we have $\sigma_k(\widetilde{U})>(2\mu-1)\|E_{I,J}\|\geq0$.   }
\end{proof}

\begin{proof}[Proof of Theorem \ref{THM:PB}]
Recalling that $\|U-\widetilde{U}_k\|\leq2\|E_{I,J}\|$ as estimated in the proof of Lemma \ref{LEM:PBTerms}, the conclusion of the theorem follows by combining this estimate with those of Proposition \ref{PROP:PB} and Lemma \ref{LEM:PBTerms}, and rearranging terms.  
\end{proof}

\subsection{Proofs for Section \ref{SEC:CkCk}}
{\color{black}To begin, we mention the following straightforward lemma of Drineas and Ipsen.

\begin{lemma}[{\cite[Theorem 2.3]{drineas2019low}}]\label{THM:DrineasIpsen}
Let $A,E\in\K^{m\times n}$ and let $P\in\K^{m\times m}$ be an orthogonal projection ($P^2=P^*=P$). Then 
\[ \|(I-P)(A+E)\| \leq \|(I-P)A\| + \|E\|. \]
\end{lemma}

Note that Lemma \ref{THM:DrineasIpsen} was stated in \cite{drineas2019low} only for $\K=\R$ and Schatten $p$--norms for integer $p$, but it is easily seen to hold for all norms of the form allowed here.}

\begin{lemma}\label{LEM:CkCkA}
The following hold:
\begin{align*}\|A-\widetilde{C}_k\widetilde{C}_k^\dagger \widetilde{A}\| &\leq 2\|E_J\|\|C^\dagger A\|+\|E\|,\\
\|A-\widetilde{A}\widetilde{R}_k^\dagger\widetilde{R}_{\color{black}k}\| &\leq 2\|E_I\|\|AR^\dagger\|+\|E\|.
\end{align*}
\end{lemma}
\begin{proof}
{\color{black}
Note that $\widetilde{C}_k\widetilde{C}_k^\dagger$ is an orthogonal projection, and since $\widetilde{C} = C+E_J$, applying Lemma \ref{THM:DrineasIpsen} directly gives
\begin{equation}\label{EQN:ICkCktilde} \|(I-\widetilde{C}_k\widetilde{C}_k^\dagger)C\| = \|(I-\widetilde{C}_k\widetilde{C}_k^\dagger)(\widetilde{C}-E_J)\|\leq \|(I-\widetilde{C}_k\widetilde{C}_k^\dagger)\widetilde{C}\|+\|E_J\|. \end{equation}
Now there are two cases to consider; if $\rank(\widetilde{C}_k)<k$ (e.g., if $\rank(\widetilde{C})<k$), then $\widetilde{C}_k=\widetilde{C}$, and the right-hand side of \eqref{EQN:ICkCktilde} is simply $\|E_J\|$; on the other hand, if $\rank(\widetilde{C}_k)=k$, then the right-hand side of \eqref{EQN:ICkCktilde} is
\[ \sigma_{k+1}(\widetilde{C}) + \|E_J\| \leq \sigma_{k+1}(C)+2\|E_J\| = 2\|E_J\|\]
by Weyl's inequality (Theorem \ref{THMHoffmanWielandt}).

Now $A=CC^\dagger A$ since $\rank(C)=\rank(A)=k$, and we have
\begin{eqnarray*}
\| A-\widetilde{C}_k\widetilde{C}_k^\dagger \widetilde{A}\|&\leq&\|A-\widetilde{C}_k\widetilde{C}_k^\dagger A\|+\|E\|\\
&=&\|(I-\widetilde{C}_k\widetilde{C}_k^\dagger)CC^\dagger A\|+\|E\|\\
&\leq&2\|E_J\|\|C^\dagger A\|+\|E\|.
\end{eqnarray*}
The second stated inequality follows from the same argument \textit{mutatis mudandis}.}
\end{proof}

\begin{proof}[Proof of Theorem \ref{THM:CkARk}]
Mimic the proof of Theorem \ref{THM:ErrFApp} while applying Lemma \ref{LEM:CkCkA}.
\end{proof}

\section{How to enforce the rank?}\label{SEC:Rank}

Our previous analysis illustrated two natural ways to enforce the rank of CUR approximations; namely, $\widetilde{C}\widetilde{U}_k^\dagger\widetilde{R}$ and $\widetilde{C}_k\widetilde{C}_k^\dagger\widetilde{A}\widetilde{R}_k^\dagger\widetilde{R}_k$.  The first has been utilized in the special case of the CUR approximation called the Nystr\"{o}m method \cite{GittensMahoney}, which is when $A$ is symmetric positive semi-definite (SPSD) and the same columns and rows are chosen (i.e., $A\approx CU^\dagger C^*$).  It has recently been suggested by some authors that a better way to enforce the rank would be to consider $(CU^\dagger R)_k$, which means to make the CUR approximation suggested by Theorem \ref{THM:Characterization}(ii), and then take its best rank $k$ approximation \cite{BeckerNystrom,TroppNystrom}.


In particular, Pourkamali-Anaraki and Becker \cite{BeckerNystrom}, Tropp et al.~\cite{TroppNystrom}, and Wang et al.~\cite{Wang2019} have discussed that when approximating a SPSD matrix $K$ using the Nystr\"{o}m method, it is better to enforce the rank \textit{after} forming the approximation rather than during the process.  Specifically, Pourkamali-Anaraki and Becker \cite{BeckerNystrom} show that  
\begin{equation}\label{EQ:Becker} \|K- (CU^\dagger C^*)_r\|_\ast\leq \|K-CU_r^\dagger C^*\|_\ast. \end{equation}
Here $\|\cdot\|_*$ is the nuclear norm, which is the Schatten 1--norm.

However, it turns out that this can fail to be true in \textit{every} Schatten $p$--norm for non-SPSD matrices.  Indeed consider the modified version of Example 2 of \cite{BeckerNystrom}, and let

\[A = \begin{bmatrix} -1 & 0 & 10\\ 0 & 1+\eps & 0\\ 10 & 0 & 100\end{bmatrix}, \]
and let $C$ be the first two columns of $A$, $R$ be the first two rows of $A$, and {\color{black}$U$} be their intersection (i.e. $U=A(1:2,1:2)$).
Clearly
\[ U_1 = \begin{bmatrix} 0 & 0\\ 0 & 1+\eps\\\end{bmatrix}, \quad\textnormal{whence}\quad U_1^\dagger = \begin{bmatrix} 0 & 0 \\ 0 & \frac{1}{1+\eps}\\\end{bmatrix}.  \]

We then have that
\[ CU_1^\dagger R = \begin{bmatrix} 0 & 0 & 0 \\ 0 & 1+\eps & 0\\ 0 & 0 & 0\\\end{bmatrix}  .\]

On the other hand,
\[ CU^\dagger R = \begin{bmatrix} -1 & 0 & 10\\ 0 & 1+\eps & 0\\ 10 & 0 & -100\\\end{bmatrix}, \]
and
\[(CU^\dagger R)_1 = \begin{bmatrix} -1 & 0 & 10\\ 0 & 0 & 0\\ 10 & 0 & -100\\\end{bmatrix}. \]

Thus 
\[ A-CU_1^\dagger R = \begin{bmatrix}  -1 & 0 & 10\\ 0 & 0 & 0\\ 10 & 0 & 100\\\end{bmatrix}, \]
but
\[A - (CU^\dagger R)_1 = \begin{bmatrix} 0 & 0 & 0\\ 0 & 1+\eps & 0\\ 0 & 0 & 200\\\end{bmatrix}. \]

The spectrum of $A-CU_1^\dagger R$ is approximately $(100.9806, 1.9806, 0)$, but the spectrum of $A-(CU^\dagger R)_1$ is $(200, 1+\eps,0)$.

For any $p\in[1,\infty]$, the Schatten $p$--norm of the first approximation is thus in the interval $[100,103]$, whereas for every $\eps\in(0,1)$, the Schatten $p$--norm of the latter approximation lies in $[200,202]$.  In particular, the analogue of \eqref{EQ:Becker} does not hold for CUR approximations of non-SPSD matrices in general.

Note this is not a universal phenomenon. For matrices which are small random perturbations of SPSD matrices, the inequality \eqref{EQ:Becker} may be valid for certain CUR decompositions, i.e., certain choices of columns and rows.

\section{Numerical Simulations}\label{SEC:Experiments}

In this section, we compare the performance of various rank-enforcement methods for different structured matrices, e.g., SPSD, symmetric matrices, general random matrices, real data matrices from the Hopkins155 motion segmentation data set \cite{Hopkins}, and structured function-related matrices. 

\subsection{Random Matrix Experiments}

\begin{experiment}\label{EXP:1}
First, we examine the  performances of the various rank-enforcement methods for an SPSD matrix which is corrupted by noise which is SPSD (the easiest case) and noise which is symmetric but not positive semi-definite.  The purpose of this basic experiment is to test how sensitive the bound of Pourkamali-Anaraki and Becker \eqref{EQ:Becker} is to perturbations.  We first generate an SPSD matrix $A\in\R^{100\times 100}$ with rank $8$ {\color{black}(in particular, $A = G_8G_8^T$ where $G$ is a random Gaussian matrix and $G_8$ its truncated SVD).}  $A$ is then perturbed by an SPSD noise matrix $E=HH^T$ or a merely symmetric noise matrix $E = (H+H^T)/2$ where $H$ is a Gaussian random matrix whose entries are 0 mean and have standard deviation $\sigma = 10^{-3}$ (experiments with other noise levels are not shown, but the qualitative behavior is the same).  
We sample $x$ columns of $\widetilde{A}$ uniformly with replacement to form $\widetilde{C}=\widetilde{A}(:,J)$, and we allow $x$ to vary from $8$ to $60$ (this is not the optimal sampling method for general matrices, but for random matrices uniform sampling suffices to illustrate the behavior of the different methods). Because of the symmetry of the problem, we set $\widetilde R=\widetilde{C}^T$ and thus $\widetilde U=\widetilde{A}(J,J)$.   In this and subsequent experiments, we then compute the following relative errors (the norm changes from experiment to experiment and is specified for each):  
\begin{itemize}
    \item $\|A- (\widetilde{C}\widetilde{U}^\dagger \widetilde{R})_k\|/\|A\|$
    \item $\|A-\widetilde{C}\widetilde{U}_k^\dagger\widetilde{R}\|/\|A\|$
    \item $\|A-\widetilde{C}_k\widetilde{C}_k^\dagger \widetilde{A}\widetilde{R}_k^\dagger\widetilde{R}_k\|/\|A\|$
    \item $\|A-\widetilde{A}_k\|/\|A\|$.
\end{itemize}

In each case, $k$ is taken to be $8$, the known underlying rank of $A$.  The column sampling procedure is repeated 20 times so that there are 20 distinct CUR approximations of each kind for each value of $x$.  Figure \ref{FIG:RE_SPD_SPD}(left) shows the averaged errors (over the 20 choices of $\widetilde{C}$ and $\widetilde{U}$) versus the number of columns for SPSD noise, and Figure \ref{FIG:RE_SPD_SPD}(right) shows the results when the random noise is not SPSD.  Error bars show the range of minimum to maximum relative error over the 20 trials.
\begin{figure}[h!]
    \centering
		\includegraphics[width=0.45\linewidth]{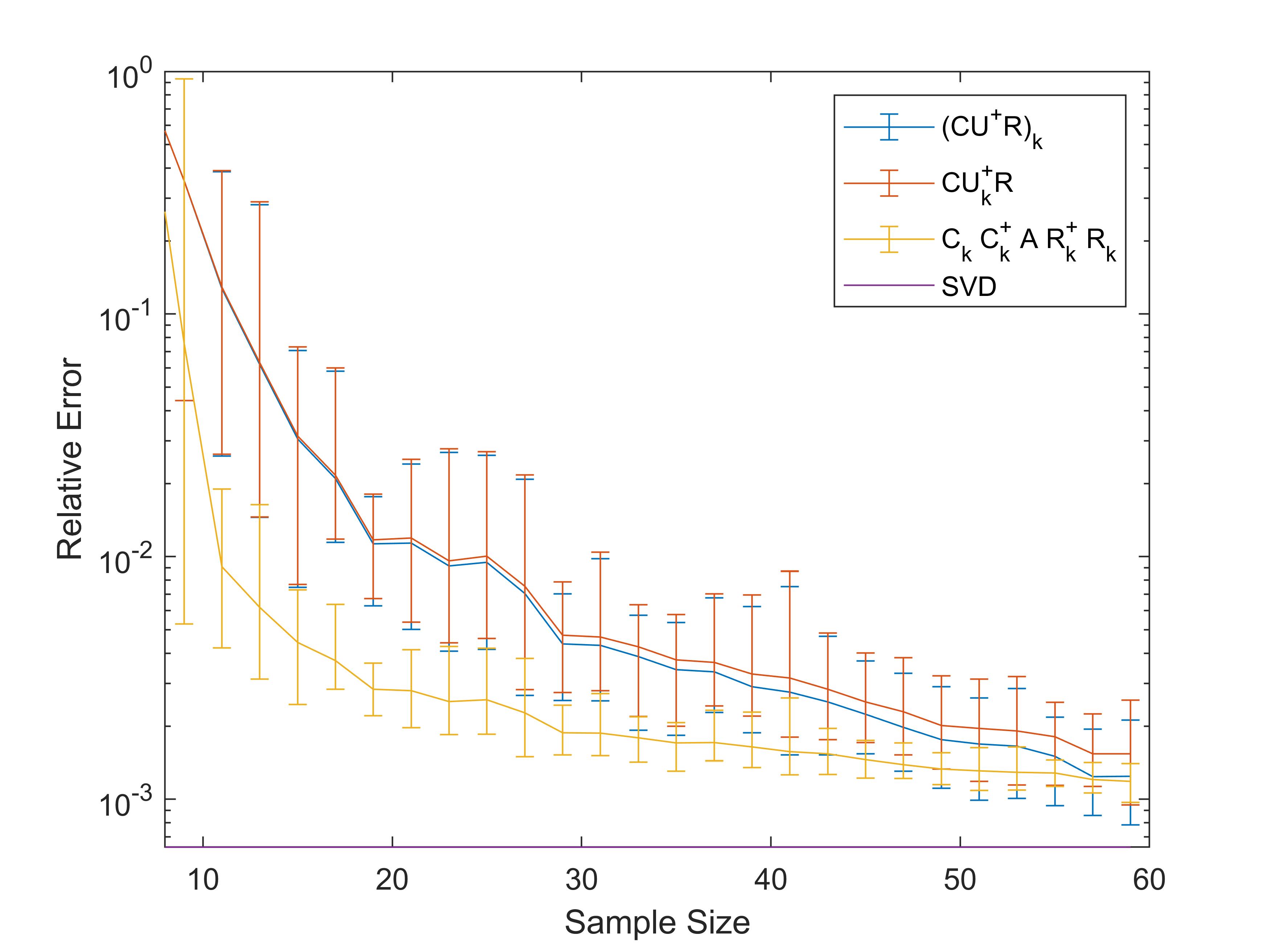}\;
		\includegraphics[width=0.45\linewidth]{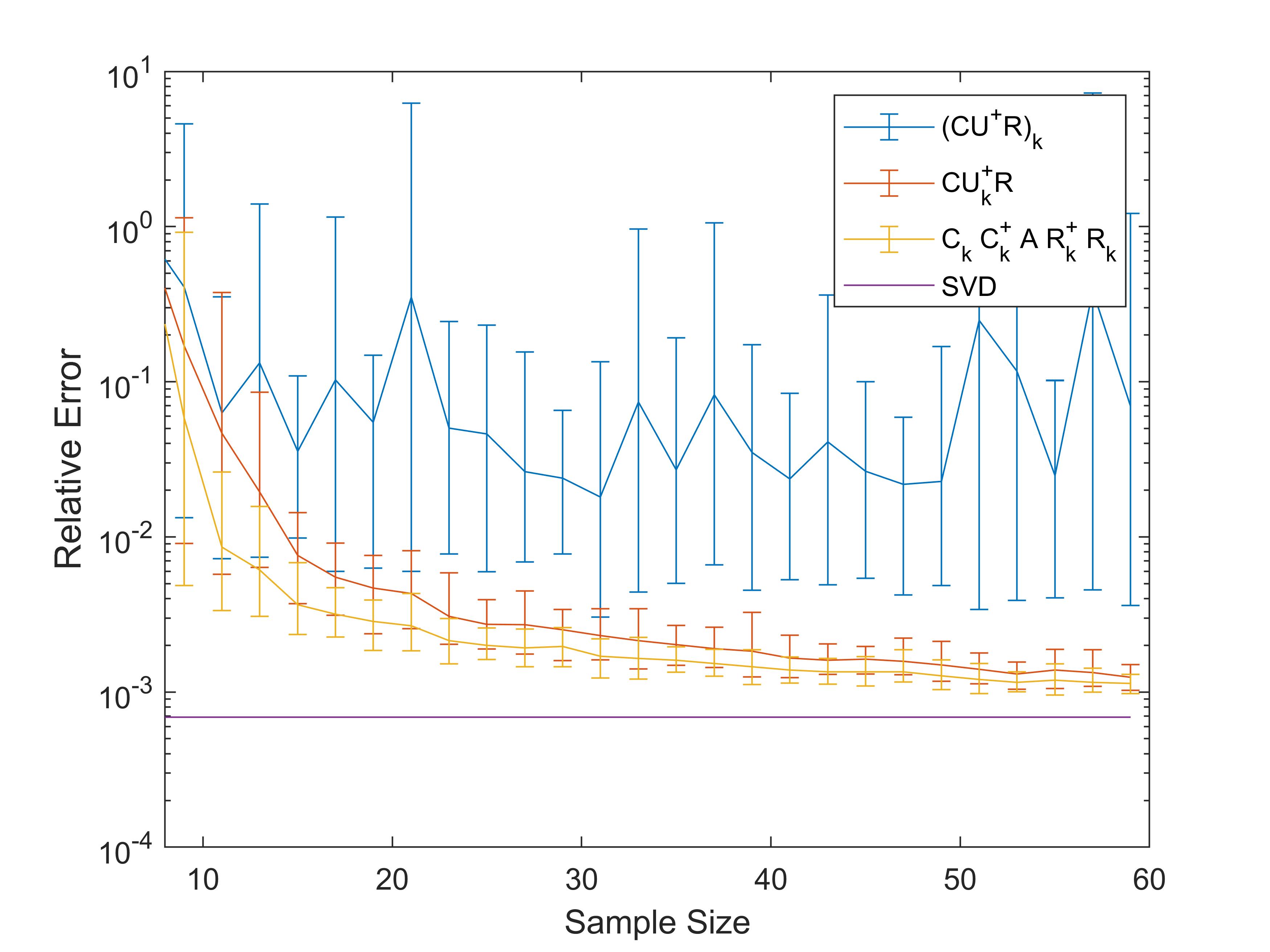}
    \caption{The performance of each rank-enforcement method for an SPSD matrix of rank $8$ plus SPSD random noise (left) and symmetric but non-PSD noise (right) with standard deviation $10^{-3}$; the plot shows relative error in the nuclear norm averaged over 20 trials vs. the number of columns selected.}\label{FIG:RE_SPD_SPD}
\end{figure}

{\color{black}The main conclusions of interest from Figure \ref{FIG:RE_SPD_SPD} are that 1) the low-rank projection based approximation $A\approx \widetilde{C}_k\widetilde{C_k}^\dagger\widetilde{A}(\widetilde{C}_k^T)^\dagger\widetilde{C}_k^T$ can yield better approximation than $A\approx (\widetilde{C}\widetilde{U}^\dagger\widetilde{C}^T)_k$ in some cases (the next experiment shows this is not always true), and 2) the better performance of the approximation $A\approx (\widetilde{C}\widetilde{U}^\dagger\widetilde{C}^T)_k$ may be highly dependent upon the SPSD structure of the underlying matrix; Figure \ref{FIG:RE_SPD_SPD}(right) shows that the performance of this approximation for an SPSD matrix with symmetric but non-PSD noise is very poor. We note for the reader that in the generic case where $A$ is a rectangular Gaussian random matrix with no other prescribed structure, the qualitative behavior of the approximations follows that of Figure \ref{FIG:RE_SPD_SPD}(right), and for brevity we do not include plots of this.}

\end{experiment}

{\color{black}
\begin{experiment} 
It is natural to ask what effect the decay of the spectrum of the underlying low-rank matrix has on the approximation. Following \cite{tropp2019streaming}, we test the different rank-enforcement methods for random matrices generated with different spectral decay (for recent work on column selection taking into account spectral decay, we refer the reader to \cite{derezinski2020improved}).  Here matrices are not SPSD.  In particular, we first generate a Gaussian random matrix $A\in\mathbb{R}^{500\times 500}$ with SVD $A=W\Sigma V^*$ and then force the matrix $\Sigma$ to have exponential or polynomial decay;  specifically 
\[\Sigma = \textbf{diag}(\underbrace{1,\cdots,1}_{10},e^{-c\cdot 11},e^{-c\cdot 12} ,\cdots,e^{-c\cdot 500})
\] or 
\[\Sigma = \textbf{diag}(\underbrace{1,\cdots,1}_{10},\frac{1}{11^c},\frac{1}{12^c},\cdots,\frac{1}{500^c}).
\]
{\color{black}Then we do the same simulations as in Experiment \ref{EXP:1}, and the results are reported in Figure \ref{FIG:RE_GM_exp} and \ref{FIG:RE_GM_pol} for various values of the parameter $c$.}


\begin{figure}[h]
    \centering
		\includegraphics[width=0.49\linewidth]{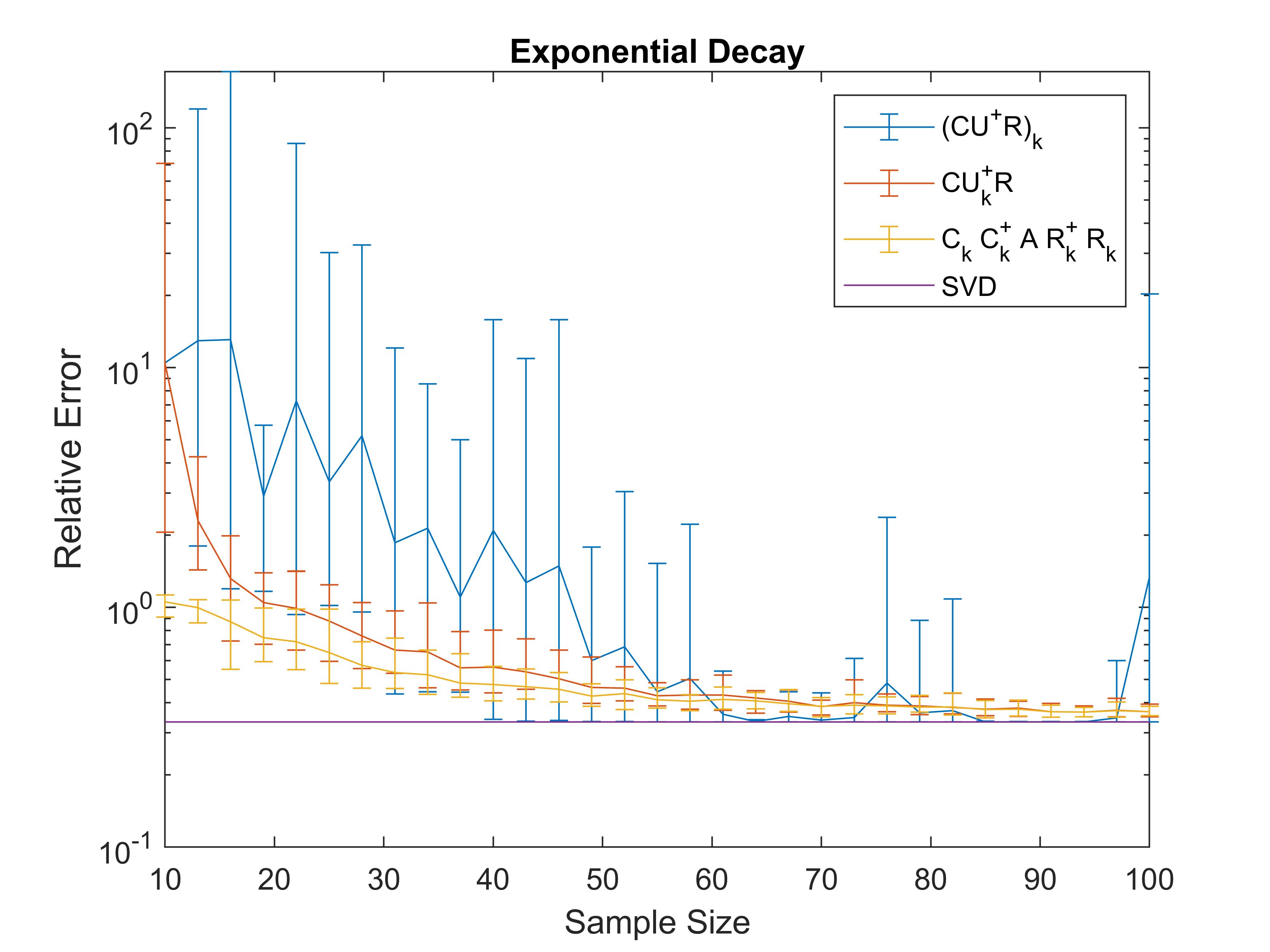}
		\includegraphics[width=0.49\linewidth]{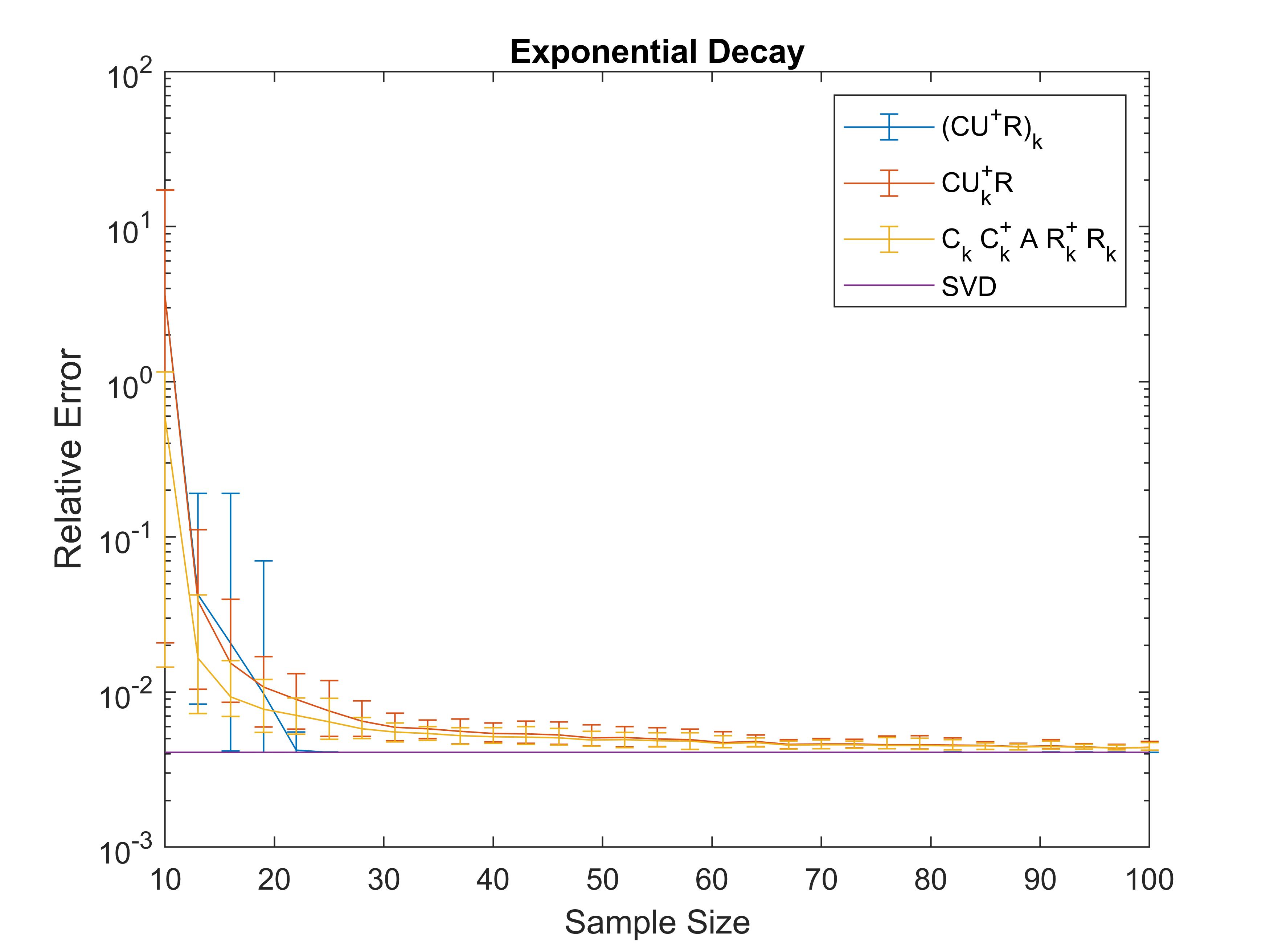}
    \caption{{\color{black}The performances of each rank-enforcement method for a random matrix whose spectrum decays exponentially with decay parameter $c=0.1$ (left) and $c=0.5$ (right).  The plot shows relative error in the spectral norm averaged over 20 trials vs. the number of columns and rows selected.}}\label{FIG:RE_GM_exp} 
\end{figure}
\begin{figure}[h]
    \centering
		\includegraphics[width=0.49\linewidth]{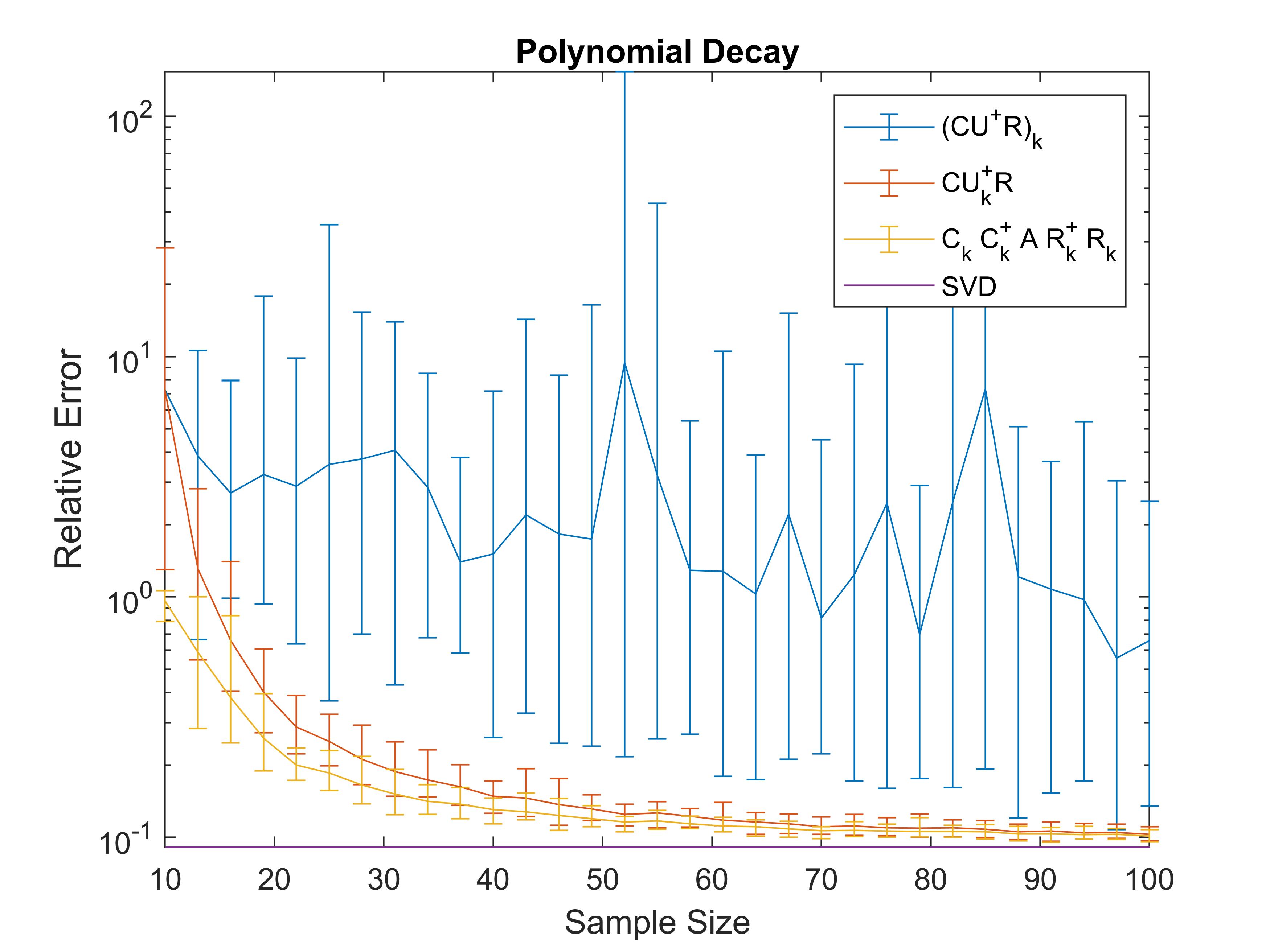}
		\includegraphics[width=0.49\linewidth]{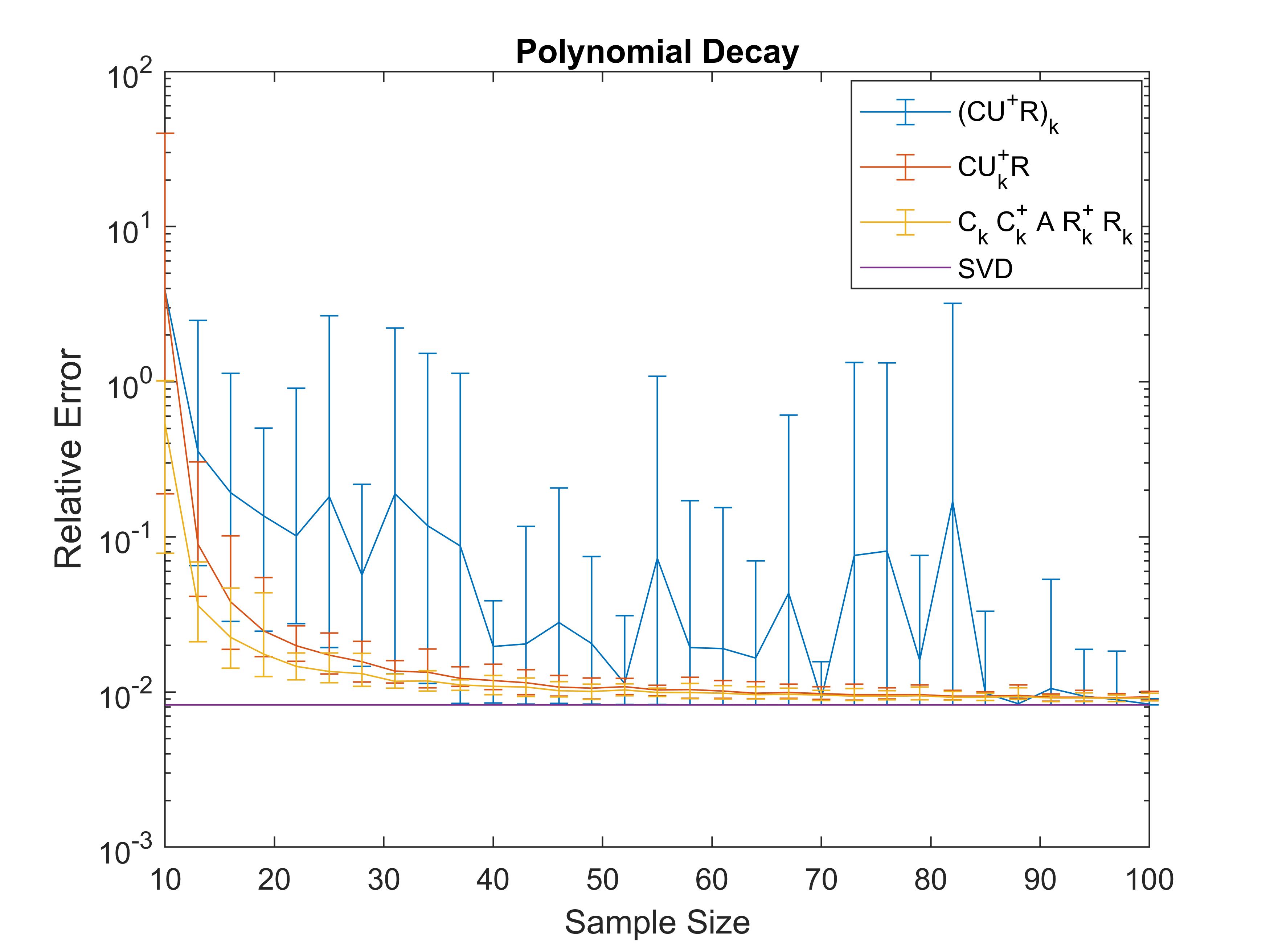}
    \caption{{\color{black}The performances of each rank-enforcement method for a random matrix whose spectrum decays polynomially with the powers $c=1$ (left) and $c=2$ (right); the plot shows relative error in the spectral norm averaged over 20 trials vs. the number of columns and rows selected.}}\label{FIG:RE_GM_pol} 
\end{figure}

Varying the exponential parameter and the power of the polynomial decay leads to qualitatively similar behavior, with faster decay of the singular values typically leading to faster decay in the relative error in terms of the sample size of columns and rows. 
\end{experiment}
}


\subsection{Deterministic and Real Data Experiments}

\begin{experiment}\label{EXP:}  In this experiment, we test the performance of the rank-enforcement methods on a deterministic matrix $B$ of size $62\times 159$, which comes from the Hopkins155 motion segmentation data set \cite{Hopkins}. The test process is the same as in Experiment \ref{EXP:1}, and the results are shown in Figure \ref{FIG:Hopkins155_GM}. 

Many times in applications, a kernel matrix (which is SPSD) is formed from the data, for example as a precursor to Spectral Clustering \cite{ng2002spectral}.  For illustration, we test the different approximation in this case, in which from $B$ above, we generate the Gaussian kernel matrix $K$ of size $159\times 159$ by setting $K_{ij}:= e^{-\|B(:,i)-B(:,j)\|^2}$. Then we repeat the process in Experiment 1 by testing the rank $40$ CUR approximation of $\widetilde A$ and choosing $x$ (the number of columns) to range from $40$ to $100$.  This value of the rank was determined empirically by analyzing the scree plot of the singular values of $K$. 

\begin{figure}[h!]
    \centering
	\includegraphics[width=.45\textwidth]{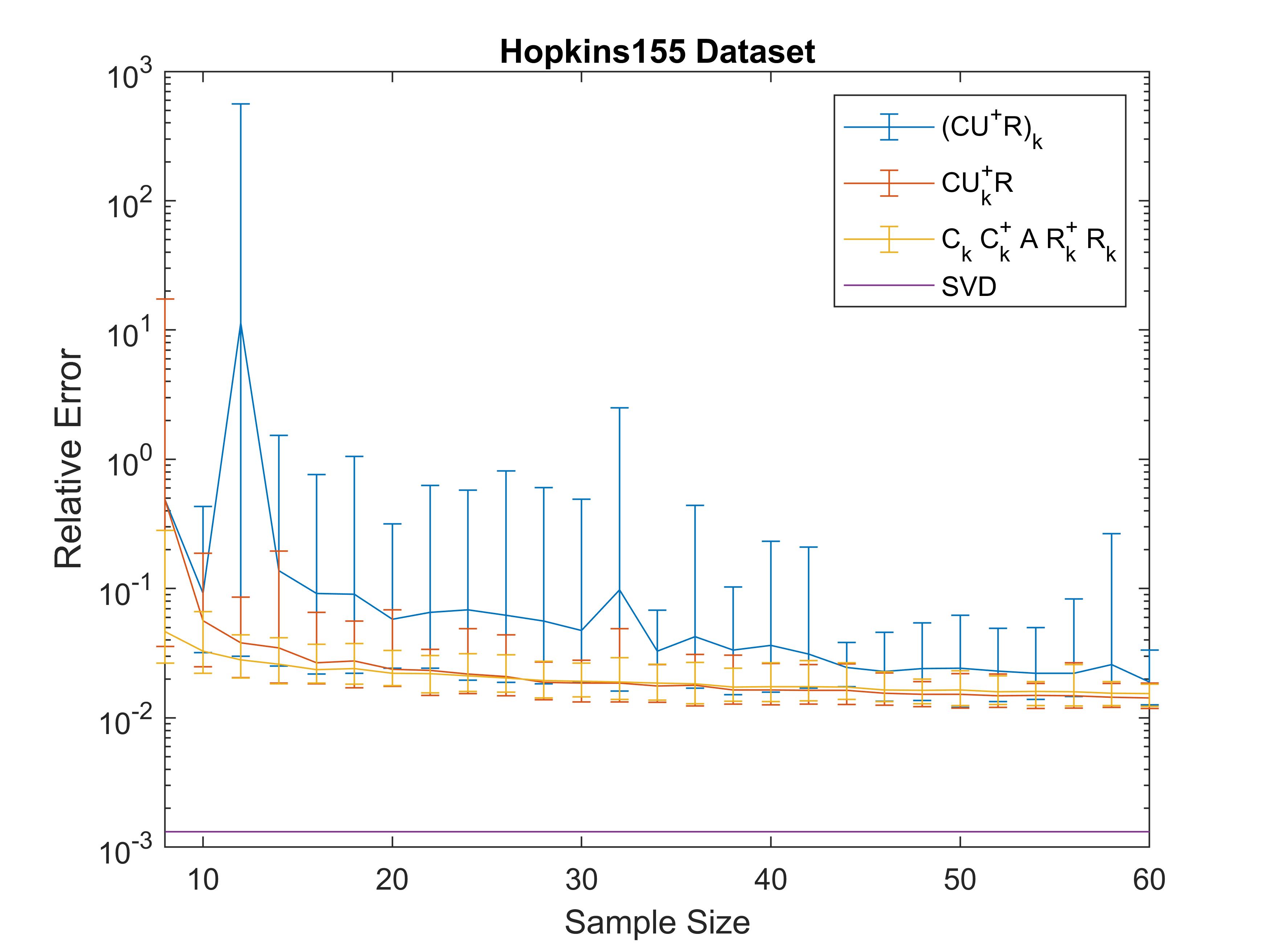}
		\includegraphics[width=0.45\linewidth]{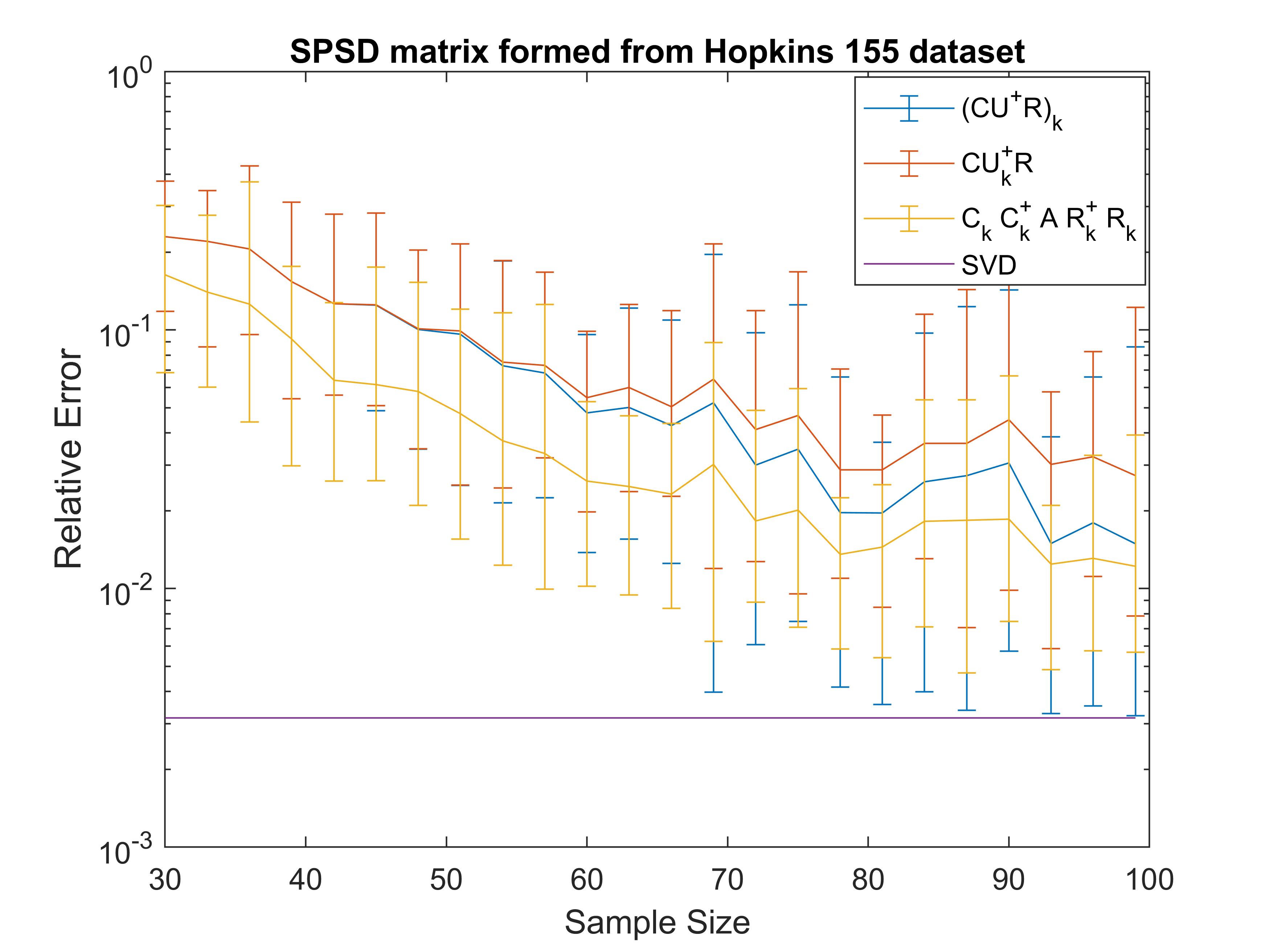}
    \caption{(Left) The rank-8 CUR approximations of the Hopkins155 data matrix showing relative error in the spectral norm vs. the number of columns and rows chosen. (Right) Error for the rank-$40$ Nystr\"{o}m approximation of the SPSD Gaussian kernel matrix $\widetilde{A}$ related to the Hopkins155 matrix; shown is relative error in the spectral norm vs. the number of columns and rows chosen.}\label{FIG:Hopkins155_GM} 
\end{figure}

{\color{black}As seen in Figure \ref{FIG:Hopkins155_GM}(left), most CUR approximations perform similarly on the raw data matrix except for that of the form $(CU^\dagger C^T)_k$ which has larger average error and variance.  For the kernel matrix (right), the rank-truncated projection method used here performs better on average than the others (though not for every instance as one sees the minimum error achieved by the approximation $(CU^\dagger C^T)_k$ nears the SVD error for large numbers of columns.  Not shown is the nuclear norm case in which
something interesting occurs; neither approximation is always better on average as they switch places in terms of performance around a choice of approximately 65 columns. We conclude that the new projection based method proposed here is neither strictly better nor worse (in either spectral or nuclear norm) than previously proposed rank-enforcement variants of the Nystr\"{o}m method.
}
\end{experiment}

\begin{experiment}\label{EXP:4}
{\color{black} To compare with other works on cross-approximation (a generalized variant of CUR) here we consider the Hilbert matrix with entries $H_{ij} = \frac{1}{i+j-1}$, which appears in various settings including classical polynomial approximation and is notoriously ill-conditioned even for small size (see \cite{choi1983tricks} for an expository article or \cite{benzi2015exploiting,tyrtyshnikov2004kronecker} for concerns closer to the current work).  We take $H$ to be of size $500\times 500$, and run essentially the same test as in Experiment \ref{EXP:1} but for different methods of column and row sampling. The primary purpose is to illustrate that the method of choosing columns and rows can have tremendous effect on the accuracy of the reconstruction in many instances; this notion has been explored in previous works (e.g., \cite{SorensenDEIMCUR,VoroninMartinsson}) but we add the additional method of Maximum Volume Sampling \cite[Algorithm 1]{mikhalev2018rectangular} of the truncated singular vectors of $H$.  Algorithm 1 of \cite{mikhalev2018rectangular} is a heuristic which attempts to find a good approximation to a submatrix of singular vectors which has the maximal volume (the volume of a rectangular matrix is the product of its singular values). The maximal volume selection scheme for CUR decompositions is described in \cite[Definition 9]{mikhalev2018rectangular}.  In addition to this method we consider sampling columns and rows from three distributions: the uniform distribution, proportional to column/row lengths, and the leverage score distribution (see \cite{hamm2019stability} for more details).

Results are shown in Figure \ref{FIG:recip_xy}, where we see that the Maximum Volume Sampling method gives more accurate and stable results than other sampling methods (with the exception that the approximation $H\approx (CU^\dagger R)_k$ behaves erratically for this method). However, selecting the rows and columns by using Maximum Volume Sampling method is computationally  expensive as it requires computing the truncated SVD of $H$ as well as the complexity of applying the Maximum Volume Sampling method on its left and right truncated singular vectors, which are $O(m|I|^2)$ and $O(n|J|^2)$, respectively.  We note that in the trials run here, the Maximal Volume Sampling method took approximately twice as long as the next most complex sampling method. Both leverage score and column/row length sampling perform decently well for sufficiently large sample size, but exhibit large variance.  It is important to note that uniform sampling performs extremely poorly on average in this case; it is known that uniform sampling provides a good CUR approximation when the singular vectors of the underlying matrix are incoherent \cite{DemanetWu}, but the matrix $H$ does not fall into this category.
} 

\begin{figure}[h!]
    \centering
	\begin{subfigure}[b]{0.49\linewidth}
		\includegraphics[width=\textwidth]{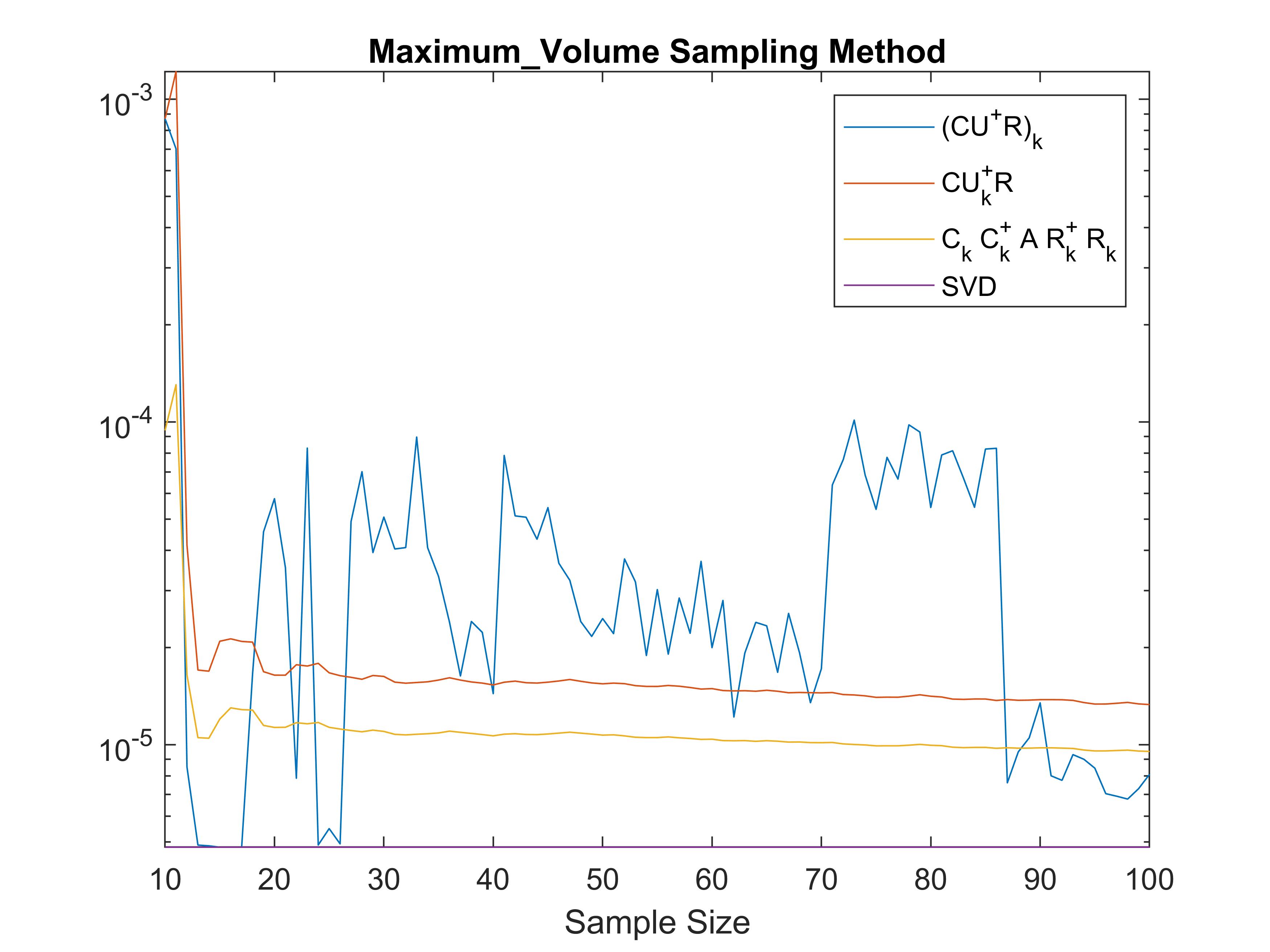}
		\caption{}
		\label{FIG:recip_xy_MV}
	\end{subfigure}
	\begin{subfigure}[b]{0.49\linewidth}
		\includegraphics[width=\textwidth]{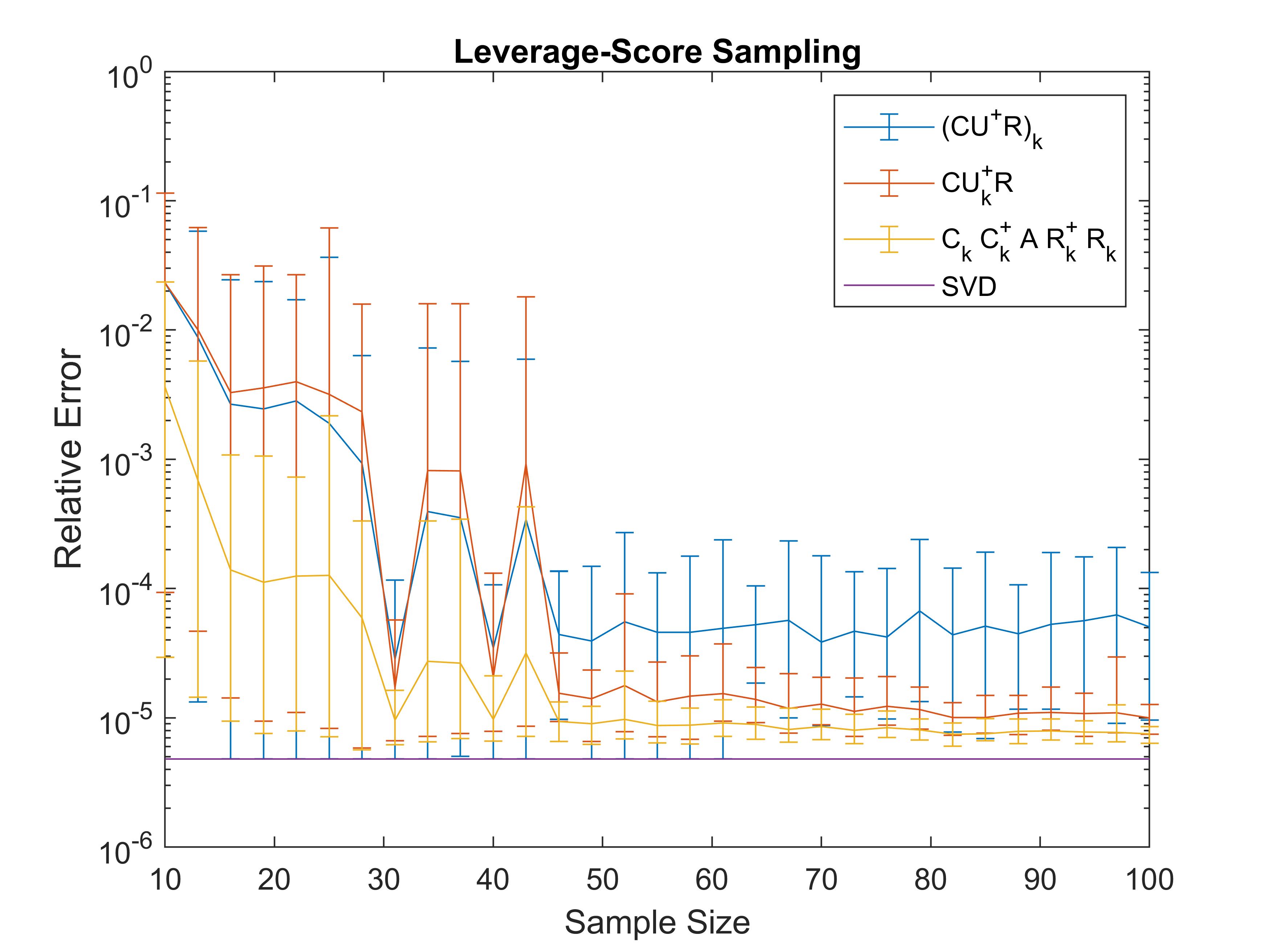}
		\caption{}
		\label{FIG:recip_xy_LS}
	\end{subfigure}
	\begin{subfigure}[b]{0.49\linewidth}
		\includegraphics[width=\textwidth]{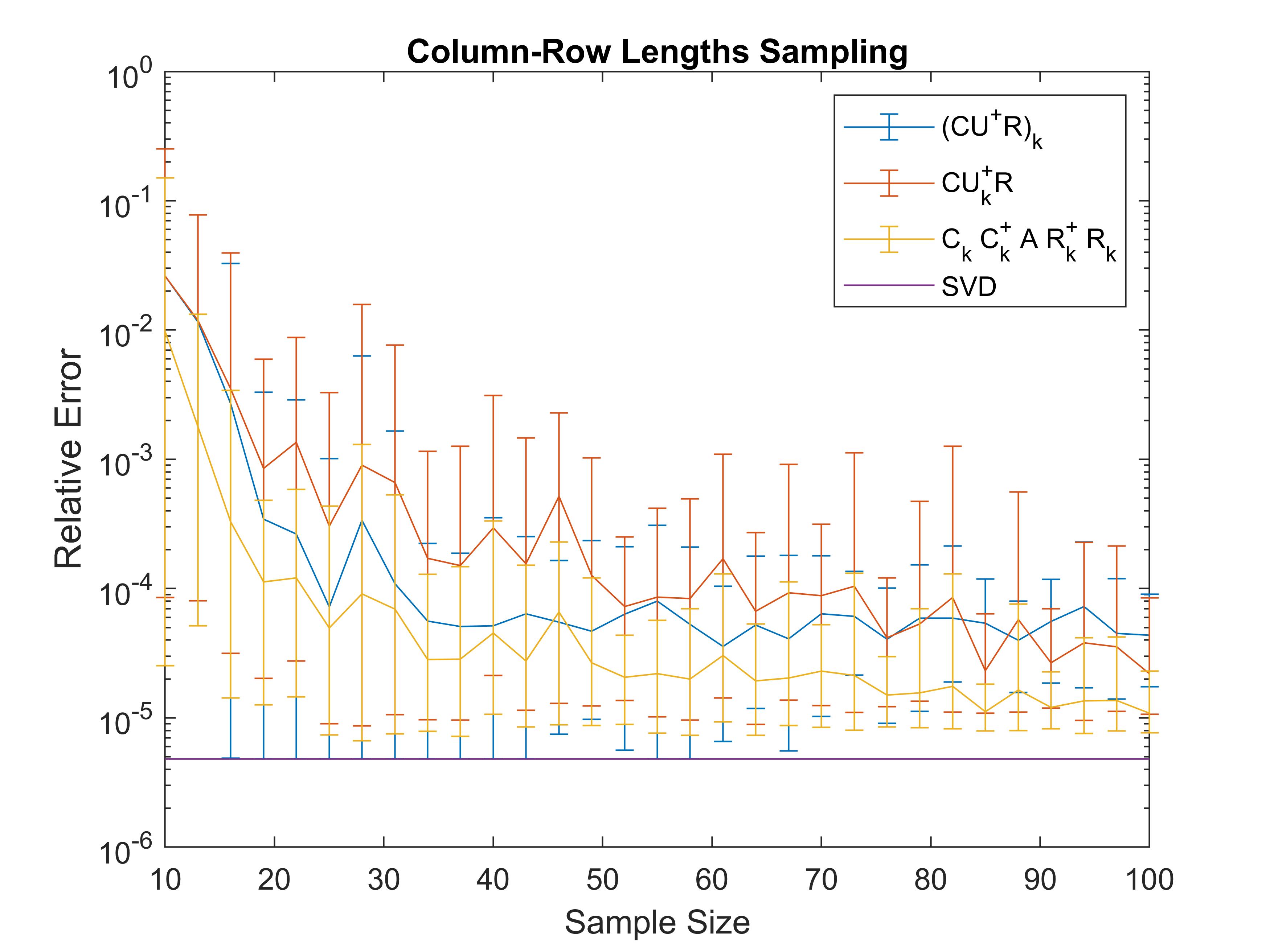}
		\caption{}
		\label{FIG:recip_xy_RCL}
	\end{subfigure}
	\begin{subfigure}[b]{0.49\linewidth}
		\includegraphics[width=\textwidth]{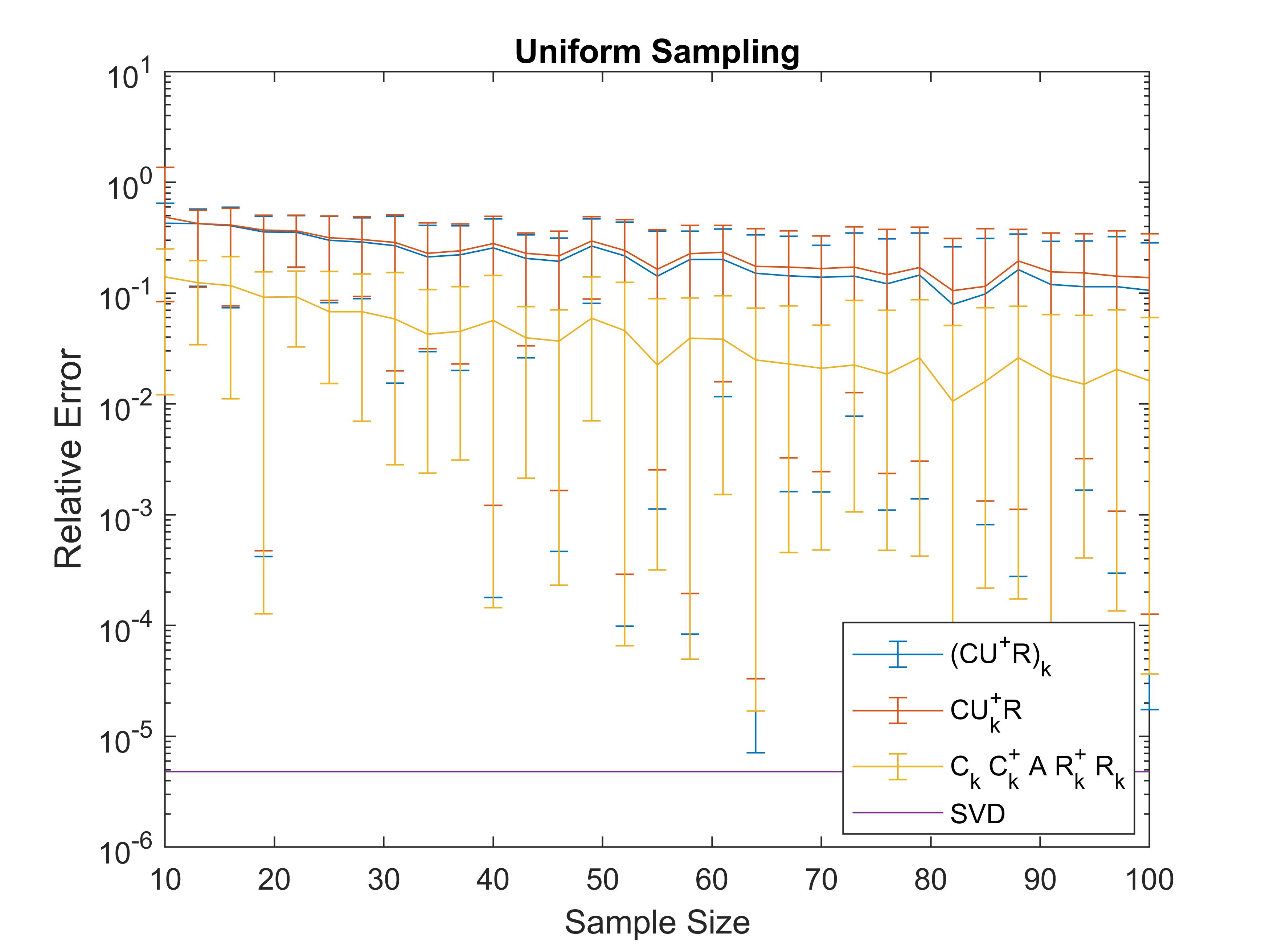}
		\caption{}
		\label{FIG:recip_xy_U}
	\end{subfigure}
    \caption{The rank-10 CUR approximations of the Hilbert matrix $H$ with $H_{ij}=\frac{1}{i+j-1}$ with different sampling patterns; the plot shows relative error in the spectral norm vs. the number of columns and rows chosen.}\label{FIG:recip_xy} 
\end{figure}

\end{experiment}

\subsection{Discussion}\label{subs:conclusion}

As seen in the experiments and figures above, the SPSD structure of matrices is crucial to the success of the rank-enforced Nystr\"{o}m method of \cite{BeckerNystrom}, i.e., of taking $A\approx (\widetilde{C}\widetilde{U}^\dagger\widetilde{C}^*)_k$ as opposed to $\widetilde{C}\widetilde{U}_k^\dagger\widetilde{C}^*$.  However, another interesting phenomenon appears, and that is that for a small oversampling of columns and rows, the new approximation introduced here of $A\approx \widetilde{C}_k\widetilde{C}_k^\dagger\widetilde{A}\widetilde{R}_k^\dagger\widetilde{R}_k$ performs better than the other rank-enforcement methods both in the SPSD and the unstructured case.  We suggest the following explanation for this: this approximation corresponds to finding the best $k$--dimensional subspace which captures the span of the columns of $\widetilde{C}$, and even when choosing few more than $k$ columns, this should be a good approximation to the span of the columns of $C$ itself.  On the other hand, the other approximations are not projections onto a $k$--dimensional subspace in the domain and range, and thus the effect of the noise on the approximation is greater.  However, as the number of columns and rows increases, the other approximations may better capture the information of $A$ by nature of better approximating the rank $k$ SVD of $\widetilde{A}$; i.e., for large $k$, $\widetilde{C}\widetilde{U}_k^\dagger\widetilde{R}\approx\widetilde{A}_k$ (this is in line with the theory known from previous works, e.g., \cite{DKMIII}; there is currently no similar theory for the projection-based approximation shown here).

{\color{black}As far as the sampling mechanism for choosing columns and rows, there is typically a tradeoff between accuracy and stability on the one hand and computational complexity on the other. Here, we see that the Maximum Volume sampling method exhibits good approximation for a small oversampling of columns and rows, whereas other sampling methods like leverage score sampling need more columns to exhibit the same accuracy, but require less computation.}

\section{Conclusion and Final Comments}\label{SEC:Conclusion}

To end, let us make some brief comments. We have provided perturbation error estimates for a variety of CUR approximation methods: estimates which hold for arbitrary matrix norms which are normalized, uniformly generated, unitarily invariant, and submultiplicative (a class which includes all Schatten $p$--norms).   Our estimates qualitatively illustrate how the column and row selections affect the error, and in particular we give some more specific bounds in the case when maximal volume submatrices are chosen. {\color{black}The estimates carried out here are of a general nature and make relatively light assumptions on the matrices involved (namely that the noise cannot be too large compared to the underlying low-rank matrix).  One can achieve better error bounds by imposing more assumptions.  In particular, assuming incoherence on the row and column spaces of $A$ can allow one to give error bounds in terms of the incoherence level as opposed to the pseudoinverse of rectangular submatrices of the singular vectors as was done here.  Additionally, if one assumes a particular method of sampling the rows and columns, then this can lead to better bounds in some instances as well.}

Due to the suggestion of other works on the Nystr\"{o}m method, we considered the effect of how the rank is enforced on CUR approximations for generic matrices, and found that, in contrast to the phenomenon observed for symmetric positive semi-definite matrices, there is no provably better way to enforce the rank for CUR approximations of arbitrary matrices.

\section*{Acknowledgements}

This research was sponsored in part by the Army Research Office and was accomplished under Grant Number W911NF-20-1-0076. The views and conclusions contained in this document are those of the authors and should not be interpreted as representing the official policies, either
expressed or implied, of the Army Research Office or the U.S. Government. The U.S. Government is authorized to reproduce and distribute reprints for Government purposes notwithstanding any copyright notation herein.
K. H. is partially supported by the NSF TRIPODS program, grant number NSF CCF--1740858.   LX.H. is partially supported by  by NSF CAREER DMS 1348721 and NSF BIGDATA 1740325.  K.H. thanks Joel Tropp, David Glickenstein, Jean-Luc Bouchot, and Vahan Huroyan for comments and suggestions on a previous version of the manuscript. {\color{black}The authors thank the anonymous reviewers for constructive feedback which helped to significantly improve the presentation of the results in this paper.}

\bibliographystyle{siamplain}
\bibliography{HammHuang}

\begin{thebibliography}{10}

\bibitem{AHKS}
{\sc A.~Aldroubi, K.~Hamm, A.~B. Koku, and A.~Sekmen}, {\em {CUR}
  decompositions, similarity matrices, and subspace clustering}, Frontiers in
  Applied Mathematics and Statistics, 4 (2019), p.~65,
  \url{https://doi.org/10.3389/fams.2018.00065},
  \url{https://www.frontiersin.org/article/10.3389/fams.2018.00065}.

\bibitem{benzi2015exploiting}
{\sc M.~Benzi and V.~Simoncini}, {\em Exploiting hidden structure in matrix
  computations: Algorithms and applications}, Springer, 2015.

\bibitem{BoutsidisOptimalCUR}
{\sc C.~Boutsidis and D.~P. Woodruff}, {\em Optimal {CUR} matrix
  decompositions}, SIAM Journal on Computing, 46 (2017), pp.~543--589.

\bibitem{bouwmans2018applications}
{\sc T.~Bouwmans, S.~Javed, H.~Zhang, Z.~Lin, and R.~Otazo}, {\em On the
  applications of robust pca in image and video processing}, Proceedings of the
  IEEE, 106 (2018), pp.~1427--1457.

\bibitem{CandesRomberg}
{\sc E.~Cand{\`e}s and J.~Romberg}, {\em Sparsity and incoherence in
  compressive sampling}, Inverse problems, 23 (2007), p.~969.

\bibitem{candes2011robust}
{\sc E.~J. Cand{\`e}s, X.~Li, Y.~Ma, and J.~Wright}, {\em Robust principal
  component analysis?}, Journal of the ACM (JACM), 58 (2011), pp.~1--37.

\bibitem{candes2009exact}
{\sc E.~J. Cand{\`e}s and B.~Recht}, {\em Exact matrix completion via convex
  optimization}, Foundations of Computational Mathematics, 9 (2009), p.~717.

\bibitem{DemanetWu}
{\sc J.~Chiu and L.~Demanet}, {\em Sublinear randomized algorithms for skeleton
  decompositions}, {SIAM} Journal on Matrix Analysis and Applications, 34
  (2013), pp.~1361--1383.

\bibitem{choi1983tricks}
{\sc M.-D. Choi}, {\em Tricks or treats with the hilbert matrix}, The American
  Mathematical Monthly, 90 (1983), pp.~301--312.

\bibitem{derezinski2020improved}
{\sc M.~Derezi{\'n}ski, R.~Khanna, and M.~W. Mahoney}, {\em Improved guarantees
  and a multiple-descent curve for the column subset selection problem and the
  nystr$\backslash$" om method}, arXiv preprint arXiv:2002.09073,  (2020).

\bibitem{drineas2019low}
{\sc P.~Drineas and I.~C. Ipsen}, {\em Low-rank matrix approximations do not
  need a singular value gap}, SIAM Journal on Matrix Analysis and Applications,
  40 (2019), pp.~299--319.

\bibitem{DKMIII}
{\sc P.~Drineas, R.~Kannan, and M.~W. Mahoney}, {\em Fast {M}onte {C}arlo
  algorithms for matrices {III}: Computing a compressed approximate matrix
  decomposition}, SIAM Journal on Computing, 36 (2006), pp.~184--206.

\bibitem{DM05}
{\sc P.~Drineas and M.~W. Mahoney}, {\em On the {N}ystr{\"o}m method for
  approximating a {G}ram matrix for improved kernel-based learning}, Journal of
  Machine Learning Research, 6 (2005), pp.~2153--2175.

\bibitem{DMM08}
{\sc P.~Drineas, M.~W. Mahoney, and S.~Muthukrishnan}, {\em Relative-error
  {CUR} matrix decompositions}, SIAM Journal on Matrix Analysis and
  Applications, 30 (2008), pp.~844--881.

\bibitem{fazel2008compressed}
{\sc M.~Fazel, E.~Candes, B.~Recht, and P.~Parrilo}, {\em Compressed sensing
  and robust recovery of low rank matrices}, in 2008 42nd Asilomar Conference
  on Signals, Systems and Computers, IEEE, 2008, pp.~1043--1047.

\bibitem{GittensMahoney}
{\sc A.~Gittens and M.~W. Mahoney}, {\em Revisiting the {N}ystr{\"o}m method
  for improved large-scale machine learning}, The Journal of Machine Learning
  Research, 17 (2016), pp.~3977--4041.

\bibitem{GolubVanLoan}
{\sc G.~H. Golub and C.~F. van Loan}, {\em Matrix Computations}, The Johns
  Hopkins University Press, Baltimore, fourth~ed., 2013.

\bibitem{goreinov2010find}
{\sc S.~A. Goreinov, I.~V. Oseledets, D.~V. Savostyanov, E.~E. Tyrtyshnikov,
  and N.~L. Zamarashkin}, {\em How to find a good submatrix}, in Matrix
  Methods: Theory, Algorithms And Applications: Dedicated to the Memory of Gene
  Golub, World Scientific, 2010, pp.~247--256.

\bibitem{Goreinov2}
{\sc S.~A. Gore\u\i{}nov, E.~E. Tyrtyshnikov, and N.~L. Zamarashkin}, {\em A
  theory of pseudoskeleton approximations}, Linear algebra and its
  applications, 261 (1997), pp.~1--21.

\bibitem{Goreinov}
{\sc S.~A. Gore\u\i{}nov, N.~L. Zamarashkin, and E.~E. Tyrtyshnikov}, {\em
  Pseudo-skeleton approximations of matrices}, Dokl. Akad. Nauk, 343 (1995),
  pp.~151--152.

\bibitem{Goreinov3}
{\sc S.~A. Gore\u\i{}nov, N.~L. Zamarashkin, and E.~E. Tyrtyshnikov}, {\em
  Pseudo-skeleton approximations by matrices of maximal volume}, Mathematical
  Notes, 62 (1997), pp.~515--519.

\bibitem{tropp}
{\sc N.~Halko, P.-G. Martinsson, and J.~A. Tropp}, {\em Finding structure with
  randomness: Probabilistic algorithms for constructing approximate matrix
  decompositions}, SIAM {R}eview, 53 (2011), pp.~217--288.

\bibitem{HH2019}
{\sc K.~Hamm and L.-X. Huang}, {\em Perspectives on {CUR} decompositions},
  Applied and Computational Harmonic Analysis, 48 (2020), pp.~1088--1099.

\bibitem{hamm2019stability}
{\sc K.~Hamm and L.-X. Huang}, {\em Stability of sampling for {CUR}
  decompositions}, Foundations of Data Science, 0 (2020), p.~0,
  \url{https://doi.org/10.3934/fods.2020006},
  \url{http://aimsciences.org//article/id/f0aa05c7-97c8-40cb-8ecf-53f2c90f7069}.

\bibitem{liu2012robust}
{\sc G.~Liu, Z.~Lin, S.~Yan, J.~Sun, Y.~Yu, and Y.~Ma}, {\em Robust recovery of
  subspace structures by low-rank representation}, IEEE Transactions on Pattern
  Analysis and Machine Intelligence, 35 (2012), pp.~171--184.

\bibitem{DMPNAS}
{\sc M.~W. Mahoney and P.~Drineas}, {\em {CUR} matrix decompositions for
  improved data analysis}, Proceedings of the National Academy of Sciences, 106
  (2009), pp.~697--702.

\bibitem{mikhalev2018rectangular}
{\sc A.~Mikhalev and I.~V. Oseledets}, {\em Rectangular maximum-volume
  submatrices and their applications}, Linear Algebra and its Applications, 538
  (2018), pp.~187--211.

\bibitem{Mirsky}
{\sc L.~Mirsky}, {\em Symmetric gauge functions and unitarily invariant norms},
  The Quarterly Journal of Mathematics, 11 (1960), pp.~50--59.

\bibitem{ng2002spectral}
{\sc A.~Y. Ng, M.~I. Jordan, and Y.~Weiss}, {\em On spectral clustering:
  Analysis and an algorithm}, in Advances in neural information processing
  systems, 2002, pp.~849--856.

\bibitem{osinsky2018rectangular}
{\sc A.~Osinsky}, {\em Rectangular maximum volume and projective volume search
  algorithms}, arXiv preprint arXiv:1809.02334,  (2018).

\bibitem{Osinsky2018}
{\sc A.~Osinsky and N.~L. Zamarashkin}, {\em Pseudo-skeleton approximations
  with better accuracy estimates}, Linear Algebra and its Applications, 537
  (2018), pp.~221--249.

\bibitem{pan2019cur}
{\sc V.~Y. Pan, Q.~Luan, J.~Svadlenka, and L.~Zhao}, {\em {CUR} low rank
  approximation of a matrix at sub-linear cost}, arXiv preprint
  arXiv:1906.04112,  (2019).

\bibitem{BeckerNystrom}
{\sc F.~Pourkamali-Anaraki and S.~Becker}, {\em Improved fixed-rank
  {N}ystr{\"o}m approximation via {QR} decomposition: Practical and theoretical
  aspects}, Neurocomputing,  (2019).

\bibitem{SorensenDEIMCUR}
{\sc D.~C. Sorensen and M.~Embree}, {\em A {DEIM} induced {CUR} factorization},
  SIAM Journal on Scientific Computing, 38 (2016), pp.~A1454--A1482.

\bibitem{Stewart_1977}
{\sc G.~W. Stewart}, {\em On the perturbation of pseudo-inverses, projections
  and linear least squares problems}, {SIAM} Review, 19 (1977), pp.~634--662,
  \url{https://doi.org/10.1137/1019104},
  \url{https://doi.org/10.1137%2F1019104}.

\bibitem{stewart1998perturbation}
{\sc G.~W. Stewart}, {\em Perturbation theory for the singular value
  decomposition}, tech. report, 1998.

\bibitem{stewart_minimizer}
{\sc G.~W. Stewart}, {\em Four algorithms for the the efficient computation of
  truncated pivoted {QR} approximations to a sparse matrix}, Numerische
  Mathematik, 83 (1999), pp.~313--323.

\bibitem{Hopkins}
{\sc R.~Tron and R.~Vidal}, {\em A benchmark for the comparison of 3-d motion
  segmentation algorithms}, in Computer Vision and Pattern Recognition, 2007.
  CVPR'07. IEEE Conference on, IEEE, 2007, pp.~1--8.

\bibitem{tropp2011improved}
{\sc J.~A. Tropp}, {\em Improved analysis of the subsampled randomized hadamard
  transform}, Advances in Adaptive Data Analysis, 3 (2011), pp.~115--126.

\bibitem{TroppNystrom}
{\sc J.~A. Tropp, A.~Yurtsever, M.~Udell, and V.~Cevher}, {\em Fixed-rank
  approximation of a positive-semidefinite matrix from streaming data}, in
  Advances in Neural Information Processing Systems, 2017, pp.~1225--1234.

\bibitem{tropp2019streaming}
{\sc J.~A. Tropp, A.~Yurtsever, M.~Udell, and V.~Cevher}, {\em Streaming
  low-rank matrix approximation with an application to scientific simulation},
  arXiv preprint arXiv:1902.08651,  (2019).

\bibitem{tyrtyshnikov2004kronecker}
{\sc E.~Tyrtyshnikov}, {\em Kronecker-product approximations for some
  function-related matrices}, Linear Algebra and its Applications, 379 (2004),
  pp.~423--437.

\bibitem{UdellTownsend2019LowRank}
{\sc M.~Udell and A.~Townsend}, {\em Why are big data matrices approximately
  low rank?}, SIAM Journal on Mathematics of Data Science, 1 (2019),
  pp.~144--160.

\bibitem{VoroninMartinsson}
{\sc S.~Voronin and P.-G. Martinsson}, {\em Efficient algorithms for {CUR} and
  interpolative matrix decompositions}, Advances in Computational Mathematics,
  43 (2017), pp.~495--516.

\bibitem{Wang2019}
{\sc S.~Wang, A.~Gittens, and M.~W. Mahoney}, {\em Scalable kernel k-means
  clustering with {N}ystr{\"o}m approximation: Relative-error bounds}, Journal
  of Machine Learning Research, 20 (2019), pp.~1--49.

\bibitem{WS_2017}
{\sc Y.~Wang and A.~Singh}, {\em Provably correct algorithms for matrix column
  subset selection with selectively sampled data}, Journal of Machine Learning
  Research, 18 (2017), pp.~5699--5740.

\end{thebibliography}

\newpage
\appendix

\section{Proof of Proposition \ref{PROP:NormTerms}}\label{APP:ProofPROP:NormTerms}
First, note that by Proposition \ref{PROP:Udagger} and the fact that $CC^\dagger=AA^\dagger$ (Lemma \ref{LEM:Projections}), we have
\[ CU^\dagger = CC^\dagger AR^\dagger = AA^\dagger AR^\dagger = AR^\dagger,\] and likewise
\[ U^\dagger R = C^\dagger A.\]
As noted in Proposition \ref{PROP:U=RAC}, we have that \[R = W_{k}(I,:)\Sigma_{k}V_{k}^*=:W_{k,I}\Sigma_{k}V_{k}^*.\] Consequently, \[AR^\dagger = W_k\Sigma_k V_k^*(W_{k,I}\Sigma_{k}V_{k}^*)^\dagger. \]
To estimate the norm, let us first notice that the pseudoinverse in question turns out to satisfy

\[(W_{k,I}\Sigma_{k}V_{k}^*)^\dagger = (V_{k}^*)^\dagger\Sigma_{k}^{-1} W_{k,I}^\dagger.\]
This is true on account of the fact that $W_{k,I}$ has full column rank, $V_{k}^*$ has orthonormal rows, and $\Sigma_{k}$ is invertible by assumption.  Next, note that since $V_k^*$ has orthonormal rows, $(V_k^*)^\dagger = V_k$.  Putting these observations together, we have that

\begin{align}\label{EQARNorm}
\|AR^\dagger\| & = \|W_k\Sigma_k V_k^*(W_{k,I}\Sigma_{k}V_{k}^*)^\dagger\| \nonumber\\
& = \|\Sigma_kV_k^*V_{k}\Sigma_{k}^{-1} W_{k,I}^\dagger\| \nonumber\\
& = \|\Sigma_k\Sigma_{k}^{-1}W_{k,I}^\dagger\| \nonumber\\
& = \|W_{k,I}^\dagger\|.
\end{align} 
The second equality follows from the unitary invariance of the norm in question; to see this, write $W_k = WP$, where $W$ is the $m\times m$ orthonormal basis from the full SVD of $A$, and $P = \begin{bmatrix}I_{k\times k}\\ 0\end{bmatrix}$; subsequently, the norm in question will be the norm of $\begin{bmatrix}W_{k,I}^\dagger\\ 0\end{bmatrix}$, which is $\|W_{k,I}^\dagger\|$.  A word of caution: Equation \eqref{EQARNorm} is not true if $W_{k,I}$ is replaced by $W_{A,I}$ the row submatrix of the full left singular vector matrix of $A$.

By a directly analogous calculation, we have that
\[
\|C^\dagger A\|= \|(V_{k,J}^*)^\dagger\|,
\]
whereupon the conclusion follows from the fact that $(V_{k,J}^*)^\dagger = (V_{k,J}^\dagger)^*$, which has the same norm as $V_{k,J}^\dagger$.

\section{Table of Inequalities}\label{APP:Table}

\begin{table}[h!]
    \centering
    \begin{tabular}{|c|c|}  
    \hline
         \textbf{Approximation} & \textbf{Error Bound} ($w=\|W_{k,I}^\dagger\|, v=\|V_{k,J}^\dagger\|$)
         \\\hline\hline
         $ \widetilde{C}\widetilde{C}^\dagger\widetilde{A}\widetilde{R}^\dagger\widetilde{R}$ & $(w+v+3)\|E\|$ \\\hline 
         
         $\widetilde{C}\widetilde{U}^\dagger\widetilde{R}$ & $(w+v+3wv)\|E\|+\|\widetilde{U}^\dagger\|(w+v+wv+1)\|E\|^2$ \\ \hline 
         
         $\widetilde{C}[\widetilde{U}]_\tau^\dagger\widetilde{R}$ & 
         $(w+v+2wv)\|E\|+wv\|[\widetilde{U}]_\tau-U\|+\|[\widetilde{U}]_\tau^\dagger\|_2(w+v+wv+1)\|E\|^2$ \\\hline 
         
         $\widetilde{C}\widetilde{U}_k^\dagger\widetilde{R}$ & 
         $(w+v+4wv)\|E\|+\dfrac{\|U^\dagger\|_2}{1-2\mu\|U^\dagger\|_2\|E\|}(w+v+wv+1)\|E\|^2$ \\\hline
         
         $\widetilde{C}_k\widetilde{C}_k^\dagger\widetilde{A}\widetilde{R}_k^\dagger\widetilde{R}$ & 
         $(2w+2v+3)\|E\|$ \\\hline
    \end{tabular}
    \caption{Summary of the perturbation bounds attained in our analysis.}
    \label{TAB}
\end{table}

\end{document}